\documentclass{amsart}
\usepackage{amsmath}
  \usepackage{paralist}
  \usepackage{amssymb}
 \usepackage{amsthm}
\usepackage[colorlinks=true]{hyperref}
\hypersetup{urlcolor=blue, citecolor=red}

\newtheorem{theorem}{Theorem}[section]

\newtheorem{lemma}[theorem]{Lemma}
\newtheorem{proposition}[theorem]{Proposition}
\theoremstyle{definition}
\newtheorem{definition}[theorem]{Definition}
\newtheorem{remark}[theorem]{Remark}

\numberwithin{equation}{section}

\title[Global existence and blowup]
      {Global existence and blowup for Choquard equations with an inverse-square potential}

\author[Xinfu Li]{}

 \keywords{well-posedness; blowup; virial identity; inverse-square potential; Choquard equation.}

\thanks{Email Address:  lxylxf@tjcu.edu.cn (XL)}

\begin{document}
\maketitle

\centerline{\scshape Xinfu Li}
\medskip
{\footnotesize
 \centerline{School of Science, Tianjin University of Commerce, Tianjin 300134,
P. R. China}}

\bigskip

\begin{abstract}
In this paper, the Choquard equation with an inverse-square
potential and both focusing and defocusing nonlinearities in the
energy-subcritical regime is investigated. For all the cases, the
local well-posedness result in $H^1(\mathbb{R}^N)$ is established.
Moreover, the global existence result for arbitrary initial values
is  proved in the defocusing case while  a global
existence/blowup dichotomy below the ground state is established  in the focusing case.\\
\textbf{2010 Mathematics Subject Classification}: 35Q55, 35B30,
35B44, 35B08.
\end{abstract}

\section{Introduction and main results}

\setcounter{section}{1}
\setcounter{equation}{0}

In this paper, we consider the Cauchy problem for  the Choquard
equation with an inverse-square potential
\begin{equation*}\mathrm{(CH_b)}\qquad
\begin{cases}
i\partial_tu-\mathcal{L}_bu=-a(I_{\alpha}\ast|u|^{p})|u|^{p-2}u,
\quad (t,x)\in
\mathbb{R}\times\mathbb{R}^{N},\\
u(0,x)=u_0(x),\quad x\in \mathbb{R}^{N},
\end{cases}
\end{equation*}
where $u: \mathbb{R}\times\mathbb{R}^{N}\to \mathbb{C}$, $u_0:
\mathbb{R}^{N}\to \mathbb{C}$,  $N\geq 3$, $\alpha\in(0,N)$,
$(N+\alpha)/N<p<(N+\alpha)/(N-2)$,  $I_{\alpha}$ is the Riesz
potential defined for every $x\in \mathbb{R}^N \setminus \{0\}$ by
\begin{equation*}
I_{\alpha}(x)=\frac{\Gamma(\frac{N-\alpha}{2})}{\Gamma(\frac{\alpha}{2})\pi^{N/2}2^\alpha|x|^{N-\alpha}}:=\frac{\mathcal{A}}{|x|^{N-\alpha}}
\end{equation*}
with $\Gamma$ denoting the Gamma function (see \cite{Riesz1949AM}),
$a=1$ (focusing case) or $-1$ (defocusing case), and $\mathcal{L}_b$
is an inverse-square potential. More precisely,
\begin{equation*}
\mathcal{L}_b:=-\Delta +b|x|^{-2}\ \mathrm{with}\ b>-(N-2)^2/4,
\end{equation*}
and we consider the Friedrichs extension of the quadratic form $Q$
defined on $C_0^{\infty}(\mathbb{R}^N)\setminus\{0\}$ via
\begin{equation*}
Q(v):=\int_{\mathbb{R}^N}|\nabla v(x)|^2+\frac{b}{|x|^2}|v(x)|^2dx.
\end{equation*}
The choice of the Friedrichs extension is natural from a physical
point of view; furthermore, when $b=0$, $\mathcal{L}_b$ reduces to
the standard Laplacian $-\Delta$. For more details, see for example
\cite{Kalf-Schmincke 1975}.

The restriction on $b$ guarantees the positivity of $\mathcal{L}_b$.
In fact, by using the sharp Hardy inequality
\begin{equation}\label{e1.1}
\frac{(N-2)^2}{4}\int_{\mathbb{R}^N}\frac{|v(x)|^2}{|x|^2}dx\leq
\int_{\mathbb{R}^N}|\nabla v(x)|^2dx\ \mathrm{for\ any\ } v\in
H^1(\mathbb{R}^N),
\end{equation}
we obtain that
\begin{equation*}
Q(v)=\|\sqrt{\mathcal{L}_b}v\|_{L^2}^2\sim \|\nabla v\|_{L^2}^2\
\mathrm{for\ any}\ b>-(N-2)^2/4.
\end{equation*}
In particular, the Sobolev space $\dot{H}^1(\mathbb{R}^N)$ is
isomorphic to the space $\dot{H}_b^1(\mathbb{R}^N)$ defined in terms
of $\mathcal{L}_b$ and we denote
\begin{equation*}
\|v\|_{\dot{H}_b^1}^2:=Q(v).
\end{equation*}

Solutions to ($\mathrm{CH_b}$) conserve the mass and energy (see
Section 3), defined respectively by
\begin{equation}\label{e1.2}
M(u(t)):=\int_{\mathbb{R}^N}|u(t,x)|^2dx
\end{equation}
and
\begin{equation}\label{e1.3}
\begin{split}
E_b(u(t)):&=\frac{1}{2}\int_{\mathbb{R}^N}(|\nabla
u(t,x)|^2+b|x|^{-2}|u(t,x)|^2)dx\\
&\qquad\qquad-\frac{a}{2p}\int_{\mathbb{R}^N}(I_\alpha\ast|u(t,x)|^p)|u(t,x)|^{p}
dx.
\end{split}
\end{equation}

Equation ($\mathrm{CH_b}$) enjoys  the scaling  invariance
\begin{equation*}
u(t,x)\mapsto
u_\lambda(t,x):=\lambda^{\frac{2+\alpha}{2p-2}}u(\lambda^2t,\lambda
x),\ \ \lambda>0.
\end{equation*}
Direct calculation gives that
\begin{equation*}
\|u_{\lambda}(0,x)\|_{\dot{H}^{\gamma}}=\lambda^{\gamma-\frac{N}{2}+\frac{2+\alpha}{2p-2}}\|u_{0}(x)\|_{\dot{H}^{\gamma}}.
\end{equation*}
This gives the critical Sobolev exponent
\begin{equation*}
\gamma_b=\frac{N}{2}-\frac{2+\alpha}{2p-2}.
\end{equation*}
The mass-critical case corresponds to $\gamma_b=0$ (or
$p=p_b:=1+(2+\alpha)/N$); The energy-critical case corresponds to
$\gamma_b=1$ (or $p=p^b:=(N+\alpha)/(N-2)$); And the inter-critical
case corresponds to $\gamma_b\in(0,1)$ (or $p\in (p_b,p^b)$).

Equation ($\mathrm{CH_b}$) is a nonlocal counterpart of the
Schr\"{o}dinger equation
\begin{equation*}\mathrm{(NLS_b)}\qquad
\begin{cases}
i\partial_tu-\mathcal{L}_bu=-a|u|^{q}u, \quad (t,x)\in
\mathbb{R}\times\mathbb{R}^{N},\\
u(0,x)=u_0(x),\quad x\in \mathbb{R}^{N},
\end{cases}
\end{equation*}
which has been studied extensively, see \cite{Dinh 2017} and
\cite{Lu-Miao 2018} and the references therein. For equation
($\mathrm{NLS_b}$), the mass-critical case corresponds to $q=q_b$
and the energy-critical case corresponds to $q=q^b$, where
\begin{equation*}
q_b:=\frac{4}{N}\ \mathrm{and}\ q^b:=\left\{\begin{array}{ll}
\frac{4}{N-2},&\
\mathrm{if}\ N\geq 3,\\
\infty,&\ \mathrm{if}\ N= 1,2.
\end{array}
\right.
\end{equation*}
Define
\begin{equation*}
\tilde{E}_b(v):=\frac{1}{2}\int_{\mathbb{R}^N}(|\nabla
v|^2+b|x|^{-2}|v|^2)dx-\frac{a}{q+2}\int_{\mathbb{R}^N}|v|^{q+2}dx,
\end{equation*}
\begin{equation*}
\tilde{\sigma}:=\frac{4-(N-2)q}{Nq-4},
\end{equation*}
\begin{equation*}
\tilde{H}(b):=\tilde{E}_{b\wedge 0}(\tilde{Q}_{b\wedge
0})M(\tilde{Q}_{b\wedge 0})^{\tilde{\sigma}},\ \
\tilde{K}(b):=\|\tilde{Q}_{b\wedge 0}\|_{\dot{H}_{b\wedge
0}^1}\|\tilde{Q}_{b\wedge 0}\|_{L^2}^{\sigma},
\end{equation*}
where $\tilde{Q}_b$ with $-(N-2)^2/4<b\leq 0$ is a radial ground
state to the elliptic equation
\begin{equation*}
\mathcal{L}_bQ+Q=Q^{q+1}.
\end{equation*}
We summarize parts of the results for
($\mathrm{NLS_b}$) in the following theorem.\\
\textbf{Theorem A}(\cite{{Dinh 2017},{Okazawa et al. 2012}}). Let
$N\geq 3$, $b>-(N-2)^2/4$, $a=\pm1$ and $u_0\in H^1(\mathbb{R}^N)$.

(1) If $0\leq q<q^b$, then ($\mathrm{NLS_b}$) is local well-posed.

(2) If $a=-1$ and $0\leq q<q^b$, then ($\mathrm{NLS_b}$) is global
well-posed.

(3) If $a=1$ and  $0\leq q<q_b$, then the solution to
($\mathrm{NLS_b}$) exists globally.

(4) Assume that  $a=1$ and  $q=q_b$. If
$\|u_0\|_{L^2}<\|\tilde{Q}_{b\wedge 0}\|_{L^2}$, then the solution
$u$ exists globally and $\sup_{t\in \mathbb{R}}\|u\|_{H^1}<\infty$;
If $\tilde{E}_b(u_0)<0$ and  either $xu_0\in L^2(\mathbb{R}^N)$ or
$u_0$ is radial, then the solution blows up in finite time.

(5)  Assume that $a=1$,  $q_b<q<q^b$ and
$\tilde{E}_b(u_0)M(u_0)^{\tilde{\sigma}}<\tilde{H}(b)$. If
$\|u_0\|_{\dot{H}_b^1}\|u_0\|_{L^2}^{\tilde{\sigma}}<\tilde{K}(b)$,
then the solution $u$ exists globally and
$$\|u(t)\|_{\dot{H}_b^1}\|u(t)\|_{L^2}^{\tilde{\sigma}}<\tilde{K}(b)$$
for any $t\in \mathbb{R}$; If
$\|u_0\|_{\dot{H}_b^1}\|u_0\|_{L^2}^{\tilde{\sigma}}>\tilde{K}(b)$
and either $xu_0\in L^2(\mathbb{R}^N)$ or $u_0$ is radial, then the
solution $u$ blows up in finite time and
$$\|u(t)\|_{\dot{H}_b^1}\|u(t)\|_{L^2}^{\tilde{\sigma}}>\tilde{K}(b)$$
for any $t$ in the existence time.

For equation ($\mathrm{CH_b}$), when $b=0$, ($\mathrm{CH_0}$) is
space-translation invariant. When $p=2$, ($\mathrm{CH_0}$) is called
the Hartree equation. In this case, the local well-posedness and
asymptotic behavior of the solutions were established in
\cite{Cazenave 2003} and \cite{Ginibre-Velo 1980}. The global
well-posedness and scattering for the defocusing energy-critical
problem were discussed by Miao et al. \cite{Miao-Xu 2007}. The
dynamics of the blowup solutions with minimal mass for the focusing
mass-critical problem were investigated by Miao et al. \cite{Miao-Xu
2010}. When $N\geq 3$, $\alpha=2$, $2\leq p<(N+\alpha)/(N-2)$, Genev
and Venkov \cite{Genev-Venkov 2012} studied
 the local and global
well-posedness, the existence of standing waves, the existence of
blowup solutions, and the dynamics of the blowup solutions in the
mass-critical case. For the general case $0 < \alpha < N$ and $2\leq
p < (N+\alpha)/(N-2)$,  Chen and Guo \cite{Chen-Guo 2007} studied
the existence of blowup solutions and the strong instability of
standing waves. Bonanno et al. \cite{Bonanno-d'Avenia 2014}
investigated the soliton dynamics. Feng and Yuan \cite{Feng-Yuan
2015} studied  the local and global well-posedness, finite time
blowup and the dynamics of blowup solutions. More precisely,
\cite{Feng-Yuan 2015} obtained the
following result.\\
\textbf{Theorem B}. Let $N\geq 3$, $(N-4)_+<\alpha<N$, $2\leq
p<(N+\alpha)/(N-2)$, $a=\pm 1$ and $u_0\in H^1(\mathbb{R}^N)$.

(1) ($\mathrm{CH_0}$) is local well-posed.

(2) If one of the following cases hold, (i) $a=-1$; (ii) $a=1$ and
$2\leq p<p_b$; (iii) $a=1$, $p=p_b$ and
$\|u_0\|_{L^2}<\|Q_0\|_{L^2}$, where $Q_0$ is a radial ground state
to (\ref{e3.5}) with $b=0$; Then the solution to ($\mathrm{CH_0}$)
exists globally.

(3) If $\max\{p_b,2\}\leq p<p^b$, $xu_0\in L^2(\mathbb{R}^N)$ and
one of the following cases hold, (i) $E_0(u_0)<0$; (ii) $E_0(u_0)=0$
and $\mathrm{Im}\int_{\mathbb{R}^N}\overline{u_0}x\cdot \nabla u_0
dx<0$; (iii) $E_0(u_0)>0$ and
$\mathrm{Im}\int_{\mathbb{R}^N}\overline{u_0}x\cdot \nabla u_0
dx<-\sqrt{2E_0(u_0)}\|xu_0\|_{L^2}$; Then the solution to
($\mathrm{CH_0}$) blows up in finite time.

When $b\neq0$, equation ($\mathrm{CH_b}$) is not space-translation
invariant anymore. It is known  that
$\dot{W}_b^{\gamma,q}(\mathbb{R}^N)$ is not equivalent to
$\dot{W}^{\gamma,q}(\mathbb{R}^N)$ for general $\gamma$ and $q$ (see
\cite{Killip-Miao et al 2017}), which restricts the application of
Strichartz estimates on the study of the local well-posedness and
scattering of global solutions (see \cite{Lu-Miao 2018} for the
study of the Schr\"{o}dinger equation with an inverse-square
potential). Fortunately, Okazawa et al. \cite{Okazawa et al. 2012}
formulated an improved energy method to treat equation with an
inverse-square potential. Based of the abstract theorem established
in \cite{Okazawa et al. 2012}, there has been great progress for
 equation ($\mathrm{CH_b}$) with $p=2$. For example, Okazawa et al. \cite{Okazawa et al. 2012}
obtained the local well-posedness result when $(N-4)_+\leq \alpha<N$
and $a=\pm 1$, and further the global existence result when $a=-1$
or $a=1$ and $\alpha\geq (N-2)_+$. Suzuki  studied the scattering of
the global solution when $a=-1$ and $\alpha\in(N-2,N-1)$ in
\cite{Suzuki 2017} and  the blowup result for initial value $u_0\in
H^1(\mathbb{R}^N)$ with $xu_0\in L^2(\mathbb{R}^N)$ and $E_b(u_0)<0$
in \cite{Suzuki 2014}. To our knowledge, there are not any results
for ($\mathrm{CH_b}$) with general exponent $p\neq2$.

By comparing the results for ($\mathrm{CH_b}$) and
($\mathrm{NLS_b}$), we see that the whole picture of
($\mathrm{CH_b}$) is far to be completed, even in the
 case $b=0$. For example, the sharp
global existence/blowup dichotomy in the inter-critical case and the
blowup result for radial initial values remain open. So in this
paper, we study the local well-posedness, global existence and
blowup dichotomy for equation ($\mathrm{CH_b}$) with general $p$ and
$\alpha$, and expect to obtain similar results to
($\mathrm{NLS_b}$).

Before stating our main results, we make some notations. Define
\begin{equation*}
\sigma:=\frac{N+\alpha-Np+2p}{Np-N-\alpha-2},
\end{equation*}
\begin{equation*}
\begin{split}
H(b):=E_{{b\wedge0}}(Q_{b\wedge0})\|Q_{b\wedge0}\|_{L^2}^{2\sigma},&\
K(b):=\|Q_{b\wedge0}\|_{\dot{H}_{b\wedge0}^1}\|Q_{b\wedge0}\|_{L^2}^{\sigma},\\
H(b,\mathrm{rad}):=E_{b}(Q_{b,\mathrm{rad}})\|Q_{b,\mathrm{rad}}\|_{L^2}^{2\sigma},&\
K(b,\mathrm{rad})
:=\|Q_{b,\mathrm{rad}}\|_{\dot{H}_b^1}\|Q_{b,\mathrm{rad}}\|_{L^2}^{\sigma},
\end{split}
\end{equation*}
where $Q_{b}$  with $-(N-2)^2/4<b\leq 0$ is the radial ground state
to the elliptic equation
\begin{equation}\label{e3.5}
\mathcal{L}_bQ+Q=(I_\alpha\ast|Q|^p)|Q|^{p-2}Q
\end{equation}
and $Q_{b,\mathrm{rad}}$ with $b>0$ is the radial solution to
(\ref{e3.5}) obtained in Section 4.

We begin by defining solutions to $\mathrm{(CH_b)}$.

\begin{definition}\label{def1.1}
Let $I\subset \mathbb{R}$ be an open interval containing $0$  and
$u_0\in H^1(\mathbb{R}^N)$. We call $u: I\times \mathbb{R}^N\to
\mathbb{C}$ a weak solution to $\mathrm{(CH_b)}$ if it belongs to
$L^{\infty}(K,H^1(\mathbb{R}^N))\cap
W^{1,\infty}(K,H^{-1}(\mathbb{R}^N))$ and satisfies
$\mathrm{(CH_b)}$ in the sense of
$L^{\infty}(K,H^{-1}(\mathbb{R}^N))$ for any compact $K\subset I$.
Moreover, if $u\in C(I,H^1(\mathbb{R}^N))\cap
C^{1}(I,H^{-1}(\mathbb{R}^N))$, we call it a solution to
$\mathrm{(CH_b)}$. We call $I$ the lifespan of $u$. We call $u$ a
maximal-lifespan solution if it cannot be extended to a strictly
larger interval. We call $u$ global if $I=\mathbb{R}$ and blowup in
finite time if $I\neq \mathbb{R}$.
\end{definition}

The  main results of this paper are as follows.

\begin{theorem}\label{thm local}
Let $N\geq 3$, $b>-(N-2)^2/4$ and $b\neq0$, $\alpha\in((N-4)_+,N)$,
$2\leq p<p^b$ and $a=\pm1$. Then for any $u_0\in H^1(\mathbb{R}^N)$,
there exists  a unique maximal-lifespan  solution $u\in
C(I;H^1(\mathbb{R}^N))\cap C^1(I;H^{-1}(\mathbb{R}^N))$ to
$\mathrm{(CH_b)}$. Moreover, $u$ satisfies the conservation laws
\begin{equation*}
M(u(t))=M(u_0),\ \ E_b(u(t))=E_b(u_0),\ \mathrm{for\ any\ }t\in I,
\end{equation*}
where $M$ and $E_b$ are defined in (\ref{e1.2}) and (\ref{e1.3}),
respectively.
\end{theorem}

\begin{theorem}\label{th1.2}
Let $N\geq 3$, $b>-(N-2)^2/4$ and $b\neq0$, $\alpha\in((N-4)_+,N)$
and $u\in C(I;H^1(\mathbb{R}^N))\cap C^1(I;H^{-1}(\mathbb{R}^N))$ be
the maximal-lifespan solution to $\mathrm{(CH_b)}$ with initial
value $u_0\in H^1(\mathbb{R}^N)$. If one of the following conditions
hold:

(i)  $a=-1$ and  $2\leq p<p^b$;

(ii) $a=1$ and $2\leq p<p_b$;\\
Then $u$ exists globally and $\sup_{t\in
\mathbb{R}}\|u\|_{H^1}<\infty.$
\end{theorem}

\begin{theorem}\label{thm1.3}
Let $N\geq 3$, $b>-(N-2)^2/4$ and $b\neq0$, $\alpha\in((N-4)_+,N)$,
$a=1$,  $p=p_b$ and $u\in C(I;H^1(\mathbb{R}^N))\cap
C^1(I;H^{-1}(\mathbb{R}^N))$ be the maximal-lifespan solution to
$\mathrm{(CH_b)}$ with initial value $u_0\in H^1(\mathbb{R}^N)$.

(i) If $\|u_0\|_{L^2}<\|Q_{b\wedge0}\|_{L^2}$, then $u$ exists
globally and $\sup_{t\in \mathbb{R}}\|u\|_{H^1}<\infty;$

(ii) If $E_b(u_0)<0$ and $xu_0\in L^2(\mathbb{R}^N)$, then $u$ blows
up in finite time.
\end{theorem}

\begin{remark}\label{rem1.1}
(1). The condition $E_b(u_0)<0$ in Theorem \ref{thm1.3} (ii) is a
sufficient  but  not necessary condition, see the proof in Section
5.

(2). Since we can not estimate the nonlocal nonlinearity in the
local virial identity, the case $E_b(u_0)<0$ and $u_0\in
H_r^1(\mathbb{R}^N)$ in Theorem \ref{thm1.3} (ii) is left  open.

(3). If $-\frac{(N-2)^2}{4}<b<0$, then
\begin{equation*}
u_T(t,x)=\frac{1}{(T-t)^{\frac{N}{2}}}e^{\frac{i}{T-t}-\frac{i|x|^2}{4(T-t)}}Q_{b}\left(\frac{x}{T-t}\right),\
T>0
\end{equation*}
is a solution to $\mathrm{(CH_b)}$ which blows up at time $T$ and
$\|u_T(0)\|_{L^2}=\|Q_{b}\|_{L^2}$.

(4). In the case $b>0$, in view of the radial sharp
Gagliardo-Nirenberg inequality, similarly to the proof of Theorem
\ref{thm1.3} (i), we can show that if $u_0\in H_r^1(\mathbb{R}^N)$
and $\|u_0\|_{L^2}<\|Q_{b,\mathrm{rad}}\|_{L^2}$, then the solution
to $\mathrm{(CH_b)}$ exists globally. Moreover, it is easy to show
that
\begin{equation*}
u_T(t,x)=\frac{1}{(T-t)^{\frac{N}{2}}}e^{\frac{i}{T-t}-\frac{i|x|^2}{4(T-t)}}Q_{b,\mathrm{rad}}\left(\frac{x}{T-t}\right),\
T>0
\end{equation*}
is a solution to $\mathrm{(CH_b)}$ which blows up at time $T$ and
$\|u_T(0)\|_{L^2}=\|Q_{b,\mathrm{rad}}\|_{L^2}$.
\end{remark}

\begin{theorem}\label{thm1.4}
Let $N\geq 3$, $b>-(N-2)^2/4$ and $b\neq0$, $\alpha\in((N-4)_+,N)$,
$a=1$,  $p_b<p<p^b$ and $u\in C(I;H^1(\mathbb{R}^N))\cap
C^1(I;H^{-1}(\mathbb{R}^N))$ be the maximal-lifespan solution to
$\mathrm{(CH_b)}$ with initial value $u_0\in H^1(\mathbb{R}^N)$ and
$E_b(u_0)\|u_0\|_{L^2}^{2\sigma}<H(b)$.

(i) If $\|u_0\|_{\dot{H}_b^1}\|u_0\|_{L^2}^{\sigma}<K(b)$, then $u$
exists globally and
$$\|u(t)\|_{\dot{H}_b^1}\|u(t)\|_{L^2}^{\sigma}<K(b)\ \mathrm{for\ any\ }t\in \mathbb{R};$$

(ii)  If $\|u_0\|_{\dot{H}_b^1}\|u_0\|_{L^2}^{\sigma}>K(b)$, and
either $xu_0\in L^2(\mathbb{R}^N)$ or $u_0$ is radial(in this case,
we further assume that $p< \frac{2N+6}{N+1}$), then $u$ blows up in
finite time and
$$\|u(t)\|_{\dot{H}_b^1}\|u(t)\|_{L^2}^{\sigma}>K(b)\ \mathrm{for\ any\ }t\in I.$$
\end{theorem}

\begin{remark}\label{rmk1.1}
In Theorem \ref{thm1.4} (ii), when $u_0\in H_r^1(\mathbb{R}^N)$, the
restriction $p<\frac{2N+6}{N+1}$ is added for the following reason:
Roughly speaking, the nonlocal nonlinearity is of order $|u|^{2p}$,
so it can be controlled by $\|u\|_{\dot{H}_b^1}^2$ only if $p$ is in
a subset of $2\leq p<p^b$, see (\ref{e8.11}) and (\ref{e8.10}).
\end{remark}

In view of the radial sharp Gagliardo-Nirenberg inequality,
similarly to the proof of Theorem \ref{thm1.4}, we obtain the
following result.

\begin{theorem}\label{thm1.5}
Let $N\geq 3$, $\alpha\in((N-4)_+,N)$, $p_b<p<p^b$, $b>0$, $a=1$ and
$u\in C(I;H^1(\mathbb{R}^N))\cap C^1(I;H^{-1}(\mathbb{R}^N))$ be the
maximal-lifespan solution to $\mathrm{(CH_b)}$ with initial value
$u_0\in H_r^1(\mathbb{R}^N)$ and
$E_b(u_0)\|u_0\|_{L^2}^{2\sigma}<H(b,\mathrm{rad})$.

(i) If
$\|u_0\|_{\dot{H}_b^1}\|u_0\|_{L^2}^{\sigma}<K(b,\mathrm{rad})$,
then $u$ exists globally and
$$\|u(t)\|_{\dot{H}_b^1}\|u(t)\|_{L^2}^{\sigma}<K(b,\mathrm{rad})\ \mathrm{for\ any\ }t\in \mathbb{R};$$

(ii)  If
$\|u_0\|_{\dot{H}_b^1}\|u_0\|_{L^2}^{\sigma}>K(b,\mathrm{rad})$ and
$p_b<p< \frac{2N+6}{N+1}$, then $u$ blows up in finite time and
$$\|u(t)\|_{\dot{H}_b^1}\|u(t)\|_{L^2}^{\sigma}>K(b,\mathrm{rad}) \ \mathrm{for\ any\ }t\in I.$$
\end{theorem}

\begin{remark}\label{rek1.2}
Since $C_{GN}(b,\mathrm{rad})<C_{GN}(b)$ for $b>0$, we see from
Section 4 that $H(b)<H(b,\mathrm{rad})$ and
$K(b)<K(b,\mathrm{rad})$. This shows that the class of radial
solution enjoys strictly larger thresholds for the global
existence/blowup dichotomy.
\end{remark}

Our arguments parallel those of \cite{Dinh 2017}, where the
Schr\"{o}dinger equation  with an inverse-square potential was
considered. New technical obstructions appear in our arguments, due
to the nonlocal nonlinearity and the fast decay of the potential.
The first difficulty we face is the local well-posedness. The usual
ways to show the local well-posedness in $H^1(\mathbb{R}^N)$ are the
Kato's method and the energy method. In the presence of the singular
potential $b|x|^{-2}$, for the homogeneous Sobolev spaces
$\dot{W}_b^{q,\gamma}(\mathbb{R}^N)$ and the usual ones
$\dot{W}^{q,\gamma}(\mathbb{R}^N)$ are equivalent only in a certain
range of $\gamma$ and $q$, the Kato's method does not allow us to
study  $\mathrm{(CH_b)}$ in the energy space with the full range
$b>-\frac{(N-2)^2}{4}$. Moreover, Okazawa-Suzuki-Yokota
\cite{Okazawa et al. 2012} pointed out that the energy method
developed by Cazenave is not enough to study $\mathrm{(CH_b)}$ in
the energy space. So they formulated an improved energy method to
treat equation with an inverse-square potential and established an
abstract theorem. Based of which, \cite{Suzuki 2013} studied a
nonlocal equation, taking $\mathrm{(CH_b)}$ with $p=2$ as a special
example. Motivated by \cite{Okazawa et al. 2012} and \cite{Suzuki
2013}, in this paper, we further use this method to study
$\mathrm{(CH_b)}$ with general $p$, which needs much more
complicated calculations and an important inequality from
\cite{Cao-Li-Luo 2015}. The global existence is a direct result of
the local well-posedness and the sharp Gagliardo-Nirenberg
inequality. To show the blowup phenomenon the virial identity plays
an important role, which has not been proved in the presence of the
fast decay potential. This is the second difficulty we encounter. By
examining the proof of the virial identity in Proposition 6.5.1 in
\cite{Cazenave 2003}, the $H^2(\mathbb{R}^N)$ regularity of the
solution is important. However, we only have obtained the
$H^1(\mathbb{R}^N)$ solution by using the improved energy method. In
order to  improve the regularity, the equivalence of
$\dot{W}_b^{q,\gamma}(\mathbb{R}^N)$ and
$\dot{W}^{q,\gamma}(\mathbb{R}^N)$ is needed, so we can only
 obtain the result in parts of the  range
$b>-\frac{(N-2)^2}{4}$. So this method is not effective. Motivated
by \cite{Suzuki 2014}, we consider a proximation problem
$\mathrm{(CH_b^\delta)}$ of $\mathrm{(CH_b)}$. By using the results
and the methods of \cite{Cazenave 2003}, we can obtain the virial
identity for the solution $u_\delta$ to $\mathrm{(CH_b^\delta)}$,
and then by letting $\delta\to 0$, we obtain the virial identity for
the solution $u$ to $\mathrm{(CH_b)}$ avoiding the $H^2$
 regularity of $u$. We should point
 that the proof in \cite{Suzuki 2014} depends  heavily on $p=2$. Our
 proof is a modification of the methods of \cite{Cazenave
 2003} and \cite{Suzuki 2014}.

\medskip
This paper is organized as follows. In Section 2, we recall some
preliminary results related to equation ($\mathrm{CH_b}$). In
Section 3, we study the local well-posedness for equation
($\mathrm{CH_b}$) in the energy-subcritical case. In Section 4, we
first study the sharp Gagliardo-Nirenberg inequality by variational
methods, and then, based of which, we give the proofs of global
existence results. In Section 5, we establish the virial identities
and prove the blowup results.

\medskip

\textbf{Notations}. The notation $A\lesssim B$ means that $A\leq CB$
for some constant $C>0$. If $A\lesssim B\lesssim A$, we write $A\sim
B$. We write $A\wedge B=\min\{A, B\}$, $A\vee B=\max\{A, B\}$. We
use $C,\ C_1,\ C_2, \cdots$ to denote various constant which may
change from line to line. We use $L^q(I,L^r(\mathbb{R}^N))$
time-space norms defined via
\begin{equation*}
\|u\|_{L^q(I,L^r(\mathbb{R}^N))}:=\left(\int_{I}\|u(t)\|_{L^r(\mathbb{R}^N)}^q\right)^{\frac{1}{q}}
\end{equation*}
for any time-space slab $I\times \mathbb{R}^N$. We make the usual
modifications when $q$ or $r$ equals to $\infty$.  We also
abbreviate $L^q(I,L^r(\mathbb{R}^N))$ by $L_t^qL_x^r$. To shorten
formulas, we often omit $\mathbb{R}^N$ or $I\times \mathbb{R}^N$.
For $r\in[1, \infty]$, we let $r'$ denote the H\"{o}lder dual, i.e.,
the solution to $1/r+1/r'=1$. We define Sobolev spaces in terms of
$\mathcal{L}_b$ via
\begin{equation*}
\|u\|_{\dot{H}_b^{s,r}}=\|(\sqrt{\mathcal{L}_b})^{s/2}u\|_{L^r}\
\mathrm{and}\
\|u\|_{H_b^{s,r}}=\|(\sqrt{1+\mathcal{L}_b})^{s/2}u\|_{L^r}.
\end{equation*}
We abbreviate
$\dot{H}_b^{s}(\mathbb{R}^N)=\dot{H}_b^{s,2}(\mathbb{R}^N)$ and
$H_b^{s}(\mathbb{R}^N)=H_b^{s,2}(\mathbb{R}^N)$. We abbreviate
$\partial_ju=\frac{\partial u}{\partial x_j}$, $\partial_{ij}
u=\frac{\partial^2 u}{\partial x_j\partial x_j}$ and
$\partial_{i}u\partial_jv\partial_{ij}
w=\sum_{i=1}^{N}\sum_{j=1}^{N}\partial_{i}u\partial_jv\partial_{ij}
w$. $\delta_{ij}=1$ if $i=j$ and $0$ if $i\neq j$.
$H_r^1(\mathbb{R}^N)=\{u\in H^1(\mathbb{R}^N):\  u\ \mathrm{is\
radially\ symmetry}\}$. $B_R(0)=\{x\in \mathbb{R}^N:\ |x|<R\}$.
$B_R(0)^c=\mathbb{R}^N\setminus B_R(0)$. $\chi_{\Omega}(x)=1$ if
$x\in \Omega$ and $0$ if $x\not\in \Omega$. $2^*=\frac{2N}{N-2}$.

\section{Preliminaries}

\setcounter{section}{2} \setcounter{equation}{0}

The following well-known Hardy-Littlewood-Sobolev inequality  can be
found in \cite{Lieb-Loss 2001}.

\begin{lemma}\label{lem HLS}
Let $p,\ r>1$ and $0<\alpha<N$ with $1/p+(N-\alpha)/N+1/r=2$. Let
$u\in L^p(\mathbb{R}^N)$ and $v\in L^r(\mathbb{R}^N)$. Then there
exists a sharp constant $C(N,\alpha,p)$, independent of $u$ and $v$,
such that
\begin{equation*}
\left|\int_{\mathbb{R}^N}\int_{\mathbb{R}^N}\frac{u(x)v(y)}{|x-y|^{N-\alpha}}\right|\leq
C(N,\alpha,p)\|u\|_{L^p}\|v\|_{L^r}.
\end{equation*}
If $p=r=\frac{2N}{N+\alpha}$, then
\begin{equation*}
C(N,\alpha,p)=C_\alpha(N)=\pi^{\frac{N-\alpha}{2}}\frac{\Gamma(\frac{\alpha}{2})}{\Gamma(\frac{N+\alpha}{2})}\left\{\frac{\Gamma(\frac{N}{2})}{\Gamma(N)}\right\}^{-\frac{\alpha}{N}}.
\end{equation*}
\end{lemma}

\begin{remark}\label{rek1.31}
(1). By the Hardy-Littlewood-Sobolev inequality above, for any $v\in
L^s(\mathbb{R}^N)$ with $s\in(1,\frac{N}{\alpha})$, $I_\alpha\ast
v\in L^{\frac{Ns}{N-\alpha s}}(\mathbb{R}^N)$ and
\begin{equation*}
\|I_\alpha\ast v\|_{L^{\frac{Ns}{N-\alpha s}}}\leq C\|v\|_{L^s},
\end{equation*}
where $C>0$ is a constant depending only on $N,\ \alpha$ and $s$.

(2). By the Hardy-Littlewood-Sobolev inequality above and the
Sobolev embedding theorem, we obtain
\begin{equation}\label{e22.4}
\begin{split}
\int_{\mathbb{R}^N}(I_\alpha\ast|u|^p)|u|^p\leq
C\left(\int_{\mathbb{R}^N}|u|^{\frac{2Np}{N+\alpha}}\right)^{\frac{N+\alpha}{N}}
\leq C\|u\|_{H^1}^{2p}
\end{split}
\end{equation}
for any $p\in \left[\frac{N+\alpha}{N},\frac{N+\alpha}{N-2}\right]$,
where $C>0$ is a constant depending only on $N,\ \alpha$ and $p$.
\end{remark}

\begin{lemma}\label{lem3.2}
Let $N\geq 3$, $\alpha\in(0,N)$ and
$p\in\left[\frac{N+\alpha}{N},\frac{N+\alpha}{N-2}\right]$. Assume
that $\{w_n\}_{n=1}^{\infty}\subset H^1(\mathbb{R}^N)$ satisfying
$w_n\to w$ weakly in $H^1(\mathbb{R}^N)$ as $n\to\infty$, then
\begin{equation*}
(I_\alpha\ast|w_n|^p)|w_n|^{p-2}w_n\to
(I_\alpha\ast|w|^p)|w|^{p-2}w\ \mathrm{weakly\ in\
}H^{-1}(\mathbb{R}^N)\ \mathrm{as}\ n\to\infty.
\end{equation*}
\end{lemma}

\begin{proof}
By the Rellich theorem,  $w_n\to w$ strongly in
$L_{\mathrm{loc}}^r(\mathbb{R}^N)$ with $r\in [1,2^*)$. By using the
Hardy-Littlewood-Sobolev inequality,  for any $\varphi\in
C_0^{\infty}(\mathbb{R}^N)$,
\begin{equation*}
\begin{split}
&\int_{\mathbb{R}^N}(I_\alpha\ast|w_n|^p)|w_n|^{p-2}w_n\bar{\varphi}
dx
-\int_{\mathbb{R}^N}(I_\alpha\ast|w|^p)|w|^{p-2}w\bar{\varphi} dx\\
&=\int_{\mathbb{R}^N}(I_\alpha\ast(|w_n|^p-|w|^p))|w|^{p-2}w\bar{\varphi}dx\\
&\ \qquad
+\int_{\mathbb{R}^N}(I_\alpha\ast|w_n|^p)(|w_n|^{p-2}w_n-|w|^{p-2}w)\bar{\varphi} dx\\
&=\int_{\mathbb{R}^N}(I_\alpha\ast(|w|^{p-2}w\bar{\varphi}))(|w_n|^p-|w|^p)dx\\
&\ \qquad
+\int_{\mathbb{R}^N}(I_\alpha\ast|w_n|^p)(|w_n|^{p-2}w_n-|w|^{p-2}w)\bar{\varphi} dx\\
&\to 0
\end{split}
\end{equation*}
as $n\to\infty$. By the dense of $C_0^{\infty}(\mathbb{R}^N)$ in
$H^1(\mathbb{R}^N)$, we complete the proof.
\end{proof}

We recall the following radial Sobolev embedding, see
\cite{Cho-Ozawa 2009}.

\begin{lemma}\label{lem Radial Sob}
Let $N\geq 2$ and $1/2\leq s < 1$. Then for any $u\in
H_r^1(\mathbb{R}^N)$,
\begin{equation*}
\sup_{x\in \mathbb{R}^N\setminus\{0\}}|x|^{\frac{N-2s}{2}}|u(x)|\leq
C(N,s)\|u\|_{L^2}^{1-s}\|u\|_{\dot{H}^1(\mathbb{R}^N)}^{s}.
\end{equation*}
Moreover, the above inequality also holds for $N\geq 3$ and $s=1$.
\end{lemma}

We recall the Strichartz estimates for the Schr\"{o}dinger operator
with an inverse-square potential (see \cite{Bouclet-Mizutani 2017}
and \cite{Burq-Planchon-Stalke 2003}). We begin by introducing the
notion of admissible pair.

\begin{definition}\label{def 2.1}
We say that a pair $(p, q)$ is  Schr\"{o}dinger admissible, for
short $(p, q)\in S$, if
\begin{equation*}
\frac{2}{p}+\frac{N}{q}=\frac{N}{2},\ p, q\in [2,\infty]\
\mathrm{and}\  (p, q, N)\neq(2,\infty, 2).
\end{equation*}
\end{definition}

\begin{proposition}\label{pro 2.2}
Let $N\geq 3$ and $b>-(N-2)^2/4$. Then for any $(p,q),\ (a,b)\in S$,
the following inequalities hold:
\begin{equation}\label{e2.2}
\|e^{it\mathcal{L}_b}\varphi(x)\|_{L^p(\mathbb{R},L^q)}\lesssim
\|\varphi(x)\|_{L^2},
\end{equation}
\begin{equation}\label{e2.4}
\left\|\int_0^{t}e^{i(t-s)\mathcal{L}_b}F(s,x)ds\right\|_{L^p(\mathbb{R},L^q)}\lesssim
\|F(t,x)\|_{L^{a'}(\mathbb{R},L^{b'})}.
\end{equation}
Here, $(a,a')$ and $(b,b')$ are H\"{o}lder dual pairs.
\end{proposition}

We recall the convergence of the operators
$\{\mathcal{L}_b^n\}_{n=1}^{\infty}$ defined below arising from the
lack of translation symmetry for $\mathcal{L}_b$.

\begin{definition}\label{def 2.2}
For $\{x_n\}_{n=1}^{\infty}\subset \mathbb{R}^N$, define
\begin{equation*}
\mathcal{L}_b^n:=-\Delta+\frac{b}{|x+x_n|^2},\
\mathcal{L}_b^{\infty}:=\left\{\begin{array}{ll}
-\Delta+\frac{b}{|x+x_{\infty}|^2},&\ \mathrm{if}\ x_n\to
x_{\infty}\in\mathbb{R}^N,\\
-\Delta, &\ \mathrm{if}\ |x_n|\to\infty.
\end{array}\right.
\end{equation*}
\end{definition}

By definition, we have
$\mathcal{L}_b[u(x-x_n)]=[\mathcal{L}_b^nu](x-x_n)$. The operator
$\mathcal{L}_b^{\infty}$ appears as a limit of the operators
$\mathcal{L}_b^n$ in the following senses, see \cite{Killip-Miao
2017}.

\begin{lemma}\label{lem 2.3}
Let $N\geq 3$ and $b>-(N-2)^2/4$. Suppose
$\{t_n\}_{n=1}^{\infty}\subset \mathbb{R}$ satisfies $t_n\to
t_{\infty}\in \mathbb{R}$ and
$\{x_n\}_{n=1}^{\infty}\subset\mathbb{R}^N$ satisfies $x_n\to
x_{\infty}\in \mathbb{R}^N$ or $|x_n|\to\infty$. Then
\begin{equation*}
\lim_{n\to\infty}\|\mathcal{L}_b^nu-\mathcal{L}_b^{\infty}u\|_{\dot{H}^{-1}}=0\
\ \mathrm{for\ any\ }u\in \dot{H}^{1}(\mathbb{R}^N),
\end{equation*}
\begin{equation*}
\lim_{n\to\infty}\|e^{-it_n\mathcal{L}_b^n}u-e^{-it_{\infty}\mathcal{L}_b^{\infty}}u\|_{\dot{H}^{-1}}=0\
\  \mathrm{for\ any\ }u\in \dot{H}^{-1}(\mathbb{R}^N),
\end{equation*}
\begin{equation*}
\lim_{n\to\infty}\|\sqrt{\mathcal{L}_b^n}u-\sqrt{\mathcal{L}_b^{\infty}}u\|_{L^2}=0\
\  \mathrm{for\ any\ }u\in \dot{H}^{1}(\mathbb{R}^N).
\end{equation*}
Furthermore, for any $(p,q)\in S$ with $p\neq2$,
\begin{equation*}
\lim_{n\to\infty}\|e^{-it_n\mathcal{L}_b^n}u-e^{-it_{\infty}\mathcal{L}_b^{\infty}}u\|_{L^p(
\mathbb{R},L^q)}=0\ \  \mathrm{for\ any\ }u\in L^2(\mathbb{R}^N).
\end{equation*}
\end{lemma}

\section{Local well-posedness}

\setcounter{section}{3} \setcounter{equation}{0}

In this section, we study the local well-posedness for equation
($\mathrm{CH_b}$) by using the abstract theory established in
\cite{Okazawa et al. 2012} by using an improved energy method.

\textbf{3.1.  Abstract theory of nonlinear Schr\"{o}dinger
equations.}  Let $S$ be a nonnegative selfadjoint operator in a
complex Hilbert space $X$. Set $X_S := D(S^{1/2})$. Then we have the
usual triplet: $X_S\subset X = X^*\subset X_S^*$, where $*$ denotes
conjugate space. Under this setting $S$ can be extended to a
nonnegative selfadjoint operator in $X_S^*$ with domain $X_S$. Now
consider
\begin{equation}\label{e4.1}
\begin{cases}
i\partial_tu=Su+g(u), \quad (t,x)\in
\mathbb{R}\times\mathbb{R}^{N},\\
u(0,x)=u_0(x),\quad x\in \mathbb{R}^{N},
\end{cases}
\end{equation}
where  $g : X_S\to X_S^*$ is a nonlinear operator satisfying the
following conditions.

(G1) Existence of energy functional: there exists $G\in
C^1(X_S;\mathbb{R})$ such that $G' = g$, that is, given $u\in X_S$,
for any $\epsilon> 0$ there exists $\delta=\delta(u,\epsilon)> 0$
such that
\begin{equation*}
|G(u+v)-G(u)-\mathrm{Re}\langle g(u),v\rangle_{X_s^*,X_s}|\leq
\epsilon\|v\|_{X_S}
\end{equation*}
for any  $v\in X_S$ with $\|v\|_{X_S}<\delta$;

(G2) Local Lipschitz continuity: for any $M
> 0$ there exists $C(M) > 0$ such that
\begin{equation*}
\|g(u)-g(v)\|_{X_S^*}\leq C(M)\|u-v\|_{X_S}
\end{equation*}
for any $u, v\in X_S$ with $\|u\|_{X_S}$, $\|v\|_{X_S}\leq M$;

(G3) H\"{o}lder-like continuity of energy functional: given $M> 0$,
for any $\delta > 0$ there exists a constant $C_\delta(M)
> 0$ such that
\begin{equation*}
|G(u)-G(v)|\leq \delta+ C_\delta(M)\|u-v\|_{X}
\end{equation*}
for any $u, v\in X_S$ with $\|u\|_{X_S}$, $\|v\|_{X_S}\leq M$;

(G4) Gauge type condition: for any $u\in X_S$,
\begin{equation*}
\mathrm{Im}\langle g(u), u\rangle_{X_S^*,X_S}=0;
\end{equation*}

(G5) Closedness type condition: let $I\subset \mathbb{R}$ be a
bounded open interval and  $\{w_n\}_{n=1}^{\infty}$ be any bounded
sequence in $L^{\infty}(I;X_S)$ such that
\begin{equation*}
\begin{cases}
w_n(t)\to w(t)\ \ \mathrm{weakly\ in}\ X_S\  \mathrm{as}\ n\to
\infty\  \mathrm{for\ almost\ all} \ t\in I,\\
g(w_n)\to f \ \ \mathrm{weakly *\ in}\ L^{\infty}(I;X_S^*) \
\mathrm{as}\ n\to \infty,
\end{cases}
\end{equation*}
then
\begin{equation*}
\mathrm{Im}\int_{I}\langle f(t),w(t)\rangle_{X_S^*,X_S}dt=\lim_{n\to
\infty}\mathrm{Im} \int_{I}\langle
g(w_n(t)),w_n(t)\rangle_{X_S^*,X_S}dt.
\end{equation*}

Under the above assumptions on $g$, the authors in \cite{Okazawa et
al. 2012} established the following local well-posedness result for
(\ref{e4.1}).

\begin{theorem}\label{thm4.1}
 Assume that $g : X_S \to X_S^*$ satisfies (G1)-(G5). Then for any $u_0\in X_S$ with $\|u_0\|_{X_S}\leq M$ there
 exists an interval
$I_{M}\subset \mathbb{R}$ containing $0$ such that (\ref{e4.1})
admits
 a  local weak solution $u\in L^{\infty}(I_M, X_S)\cap
W^{1,\infty}(I_M, X_S^*)$. Moreover, $u\in C_{w}(I_M, X_S)$ and
$\|u(t)\|_{X}=\|u_0\|_{X}$ for any $t\in I_{M}$. Further assume the
uniqueness of local weak solutions to (\ref{e4.1}). Then
$$u\in C(I_M, X_S)\cap C^1(I_M, X_S^*),$$ and the conservation law
holds:
\begin{equation*}
 E(u(t))=E(u_0),\ \mathrm{for\ any} \ t\in I_M,
\end{equation*}
where $E$ is the energy of equation (\ref{e4.1}) defined by
\begin{equation*}
E(\varphi):= \frac{1}{2}\|S^{1/2}\varphi\|_X^2 + G(\varphi),\ \
\varphi\in  X_S.
\end{equation*}
\end{theorem}

\textbf{3.2.  Local well-posedness to ($\mathrm{CH_b}$)}. In this
subsection,  we use Theorem \ref{thm4.1} to prove Theorem \ref{thm
local}. To this end, we first give some lemmas.

\begin{lemma}\label{lem 4.1} $\mathrm{(}$\cite{Suzuki 2013}$\mathrm{)}$
Assume that $\alpha_1$, $\beta_1$, $\gamma_1$, $\rho_1\in
[1,\infty]$,  $k(x,y)\in L_x^{\beta_1}L_{y}^{\alpha_1}\cap
L_y^{\beta_1}L_{x}^{\alpha_1}$ and
\begin{equation*}
\alpha_1\leq \rho_1\leq \beta_1,\ \
\frac{1}{\alpha_1}+\frac{1}{\beta_1}+\frac{1}{\gamma_1}=1+\frac{1}{\rho_1}.
\end{equation*}
Then the operator defined by
\begin{equation*}
Kf(x):=\int_{\mathbb{R}^N}k(x,y)f(y)dy
\end{equation*}
is linear and bounded from $L^{\gamma_1}(\mathbb{R}^N)$ to
$L^{\rho_1}(\mathbb{R}^N)$. Moreover,
\begin{equation*}
\|Kf\|_{L^{\rho_1}}\leq (\|k\|_{L_x^{\beta_1}L_{y}^{\alpha_1}}\vee
\|k\|_{L_y^{\beta_1}L_{x}^{\alpha_1}})\|f\|_{L^{\gamma_1}}\
\mathrm{for\ any\ } f\in L^{\gamma_1}(\mathbb{R}^N).
\end{equation*}
\end{lemma}

\begin{lemma}\label{lem 4.2}
Assume that $\alpha_1$, $\beta_1\in [1,\infty]$, $\alpha_1\leq
\beta_1$ and $1/\alpha_1+1/\beta_1\leq 4/N$. Set
\begin{equation*}
\frac{1}{\gamma_1}:=1-\frac{1}{2}\left(\frac{1}{\alpha_1}+\frac{1}{\beta_1}\right).
\end{equation*}
Assume that $k(x,y)\in L_x^{\beta_1}L_{y}^{\alpha_1}$ is symmetric,
that is, $k(x,y)=k(y,x)$ for any $x, y\in \mathbb{R}^N$. Then, for
any $f,\ g\in L^{\gamma_1}(\mathbb{R}^N)$,
\begin{equation*}
\|gK(f)\|_{L^1}\leq
\|k\|_{L_x^{\beta_1}L_{y}^{\alpha_1}}\|f\|_{L^{\gamma_1}}\|g\|_{L^{\gamma_1}}
\end{equation*}
and
\begin{equation*}
\|K(f)\|_{L^{\gamma_1'}}\leq
\|k\|_{L_x^{\beta_1}L_{y}^{{\alpha_1}}}\|f\|_{L^{\gamma_1}}.
\end{equation*}
\end{lemma}

\begin{proof}
Applying Lemma \ref{lem 4.1} with $\rho_1=\gamma_1'$ and by using
the H\"{o}lder inequality, we obtain that
\begin{equation*}
\|gK(f)\|_{L^1}\leq \|K(f)\|_{\rho_1}\|g\|_{\rho_1'}\leq
\|k\|_{L_x^{\beta_1}L_{y}^{\alpha_1}}\|f\|_{L^{\gamma_1}}\|g\|_{L^{\rho_1'}},
\end{equation*}
which completes the proof.
\end{proof}

\begin{lemma}\label{lem condi}
Let $N\geq 3$, $\alpha\in ((N-4)_+,N)$, $2\leq p<(N+\alpha)/(N-2)$
and $a=\pm 1$. Define
\begin{equation}\label{e3.16}
g(u):=-a(I_\alpha\ast|u|^p)|u|^{p-2}u,\ u\in H^1(\mathbb{R}^N)
\end{equation}
and
\begin{equation}\label{e3.17}
G(u):=-\frac{a}{2p}\int_{\mathbb{R}^N}(I_\alpha\ast|u|^p)|u|^{p}dx,\
u\in H^1(\mathbb{R}^N).
\end{equation}
Then $g$ and $G$ satisfy conditions (G1)-(G5) as in subsection 3.1
with $S=\mathcal{L}_b$, $X=L^2(\mathbb{R}^N)$,
$X_S=H^1(\mathbb{R}^N)$ and $X_S^*=H^{-1}(\mathbb{R}^N)$.
\end{lemma}

\begin{proof}
When $p=2$, the lemma is proved in \cite{Suzuki 2013}. So in the
following, we assume that $p>2$. By the definition of $g$, (G4)
holds.

We verify (G1). By using the following inequality from
\cite{Cao-Li-Luo 2015}
\begin{equation*}
\begin{split}
&||\tilde{a}+\tilde{b}|^m-|\tilde{a}|^m-m|\tilde{a}|^{m-2}\tilde{a}\tilde{b}|\leq
C(|\tilde{a}|^{m-m^*}|\tilde{b}|^{m^*}+|\tilde{b}|^m),\\
& \tilde{a}, \tilde{b}\in \mathbb{R}, m>1, m^*=\min\{m,2\}, C
\mathrm{\ is\ independent\ of}\  \tilde{a}, \tilde{b},
\end{split}
\end{equation*}
and the Young inequality
\begin{equation*}
\tilde{a}^{\frac{1}{r}}+\tilde{b}^{\frac{1}{r'}}\leq
\frac{\tilde{a}}{r}+\frac{\tilde{b}}{r'},
\end{equation*}
we have, for $p/2\geq 2$,
\begin{equation}\label{e4.3}
\begin{split}
&(|u|^2+2\mathrm{Re}(u\bar{v})+|v|^2)^{p/2}\\
 & =
|u|^p+p|u|^{p-2}\mathrm{Re}(u\bar{v})+p/2|u|^{p-2}|v|^2\\
&\ \ \ \ +O(|u|^{p-4}(2\mathrm{Re}(u\bar{v})+|v|^2)^2+(2\mathrm{Re}(u\bar{v})+|v|^2)^{p/2})\\
& = |u|^p+p|u|^{p-2}\mathrm{Re}(u\bar{v})+O(|u|^{p-2}|v|^2+|v|^p)\\
&:= |u|^p+p|u|^{p-2}\mathrm{Re}(u\bar{v})+h(u,v),
\end{split}
\end{equation}
and for  $1<p/2<2$,
\begin{equation}\label{e4.4}
\begin{split}
&(|u|^2+2\mathrm{Re}(u\bar{v})+|v|^2)^{p/2}\\
&=
|u|^p+p|u|^{p-2}\mathrm{Re}(u\bar{v})+p/2|u|^{p-2}|v|^2\\
&\ \ \ \ \ +O(2\mathrm{Re}(u\bar{v})+|v|^2)^{p/2}\\
&=  |u|^p+p|u|^{p-2}\mathrm{Re}(u\bar{v})+O(|u|^{p/2}|v|^{p/2}+|v|^p)\\
&:= |u|^p+p|u|^{p-2}\mathrm{Re}(u\bar{v})+h(u,v).
\end{split}
\end{equation}
Note that we can write $h(u, v)$ in a uniform form for all $p>2$,
\begin{equation*}
h(u,v)=|u|^{p-r}|v|^{r}+|v|^p\ \ \mathrm{for\ some\ }1<r<p.
\end{equation*}
By using (\ref{e4.3}), (\ref{e4.4}), the Hardy-Littlewood-Sobolev
inequality, the Young inequality, the H\"{o}lder inequality, the
Sobolev imbedding theorem, and the equality
\begin{equation*}
|u+v|^p=(|u+v|^2)^{p/2}=(|u|^2+2\mathrm{Re}(u\bar{v})+|v|^2)^{p/2},
\end{equation*}
for any given $u\in H^1(\mathbb{R}^N)$, we obtain that
\begin{equation*}
\begin{split}
&|G(u+v)-G(u)-\mathrm{Re}\langle g(u),v\rangle_{H^{-1},H^1}|\\
=&\left|\frac{1}{2p}\int_{\mathbb{R}^N}
(I_\alpha\ast|u+v|^{p})|u+v|^pdx-\frac{1}{2p}\int_{\mathbb{R}^N}(I_\alpha\ast|u|^{p})|u|^pdx\right.\\
&\ \ \ \left.-\int_{\mathbb{R}^N}(I_\alpha\ast|u|^{p})|u|^{p-2}\mathrm{Re}(u\bar{v})dx\right|\\
=& \left|\frac{1}{2p}\int_{\mathbb{R}^N}
(I_\alpha\ast(|u|^p+p|u|^{p-2}\mathrm{Re}(u\bar{v})+h(u,v)))\right.\\
&\ \ \ \  \ \ \qquad\qquad \times(|u|^p+p|u|^{p-2}\mathrm{Re}(u\bar{v})+h(u,v))dx\\
&\ \
\left.-\frac{1}{2p}\int_{\mathbb{R}^N}(I_\alpha\ast|u|^{p})|u|^pdx
-\int_{\mathbb{R}^N}(I_\alpha\ast|u|^{p})|u|^{p-2}\mathrm{Re}(u\bar{v})dx\right|\\
\lesssim &\int_{\mathbb{R}^N}(I_\alpha\ast
h(u,v))(|u|^p+p|u|^{p-2}\mathrm{Re}(u\bar{v})+h(u,v))dx\\
&\ \ +\int_{\mathbb{R}^N}(I_\alpha\ast
(p|u|^{p-2}\mathrm{Re}(u\bar{v})))(p|u|^{p-2}\mathrm{Re}(u\bar{v})+h(u,v))dx\\
\lesssim&
\int_{\mathbb{R}^N}(I_\alpha\ast(|v|^p+|u|^{p-r}|v|^r))(|u|^p+|v|^p)dx\\
&\ \ \ \
+\int_{\mathbb{R}^N}(I_\alpha\ast(|u|^{p-1}|v|))(|u|^{p-1}|v|+|v|^p)dx\\
\lesssim& \|v\|_{H^1}^{2p}+\|v\|_{H^1}^r,
\end{split}
\end{equation*}
which implies that (G1) holds.

We verify (G2). By using the following inequalities
\begin{equation}\label{e3.12}
||u|^p-|v|^p|\lesssim (|u|+|v|)^{p-1}|u-v|,
\end{equation}
\begin{equation*}
||u|^{p-2}u-|v|^{p-2}v|\lesssim (|u|^{p-2}+|v|^{p-2})|u-v|,
\end{equation*}
the Hardy-Littlewood-Sobolev inequality, the H\"{o}lder inequality
and the Sobolev imbedding theorem, we have, for any $\varphi\in
H^1(\mathbb{R}^N)$,
\begin{equation*}
\begin{split}
&|\langle g(u)-g(v),\varphi\rangle_{H^{-1},H^1}|\\
=&\left|\int_{\mathbb{R}^N}(I_\alpha\ast|u|^p)|u|^{p-2}u\bar{\varphi}dx
-\int_{\mathbb{R}^N}(I_\alpha\ast|v|^p)|v|^{p-2}v\bar{\varphi}dx\right|\\
\lesssim& \left|\int_{\mathbb{R}^N}(I_\alpha\ast|u|^p)(|u|^{p-2}u-|v|^{p-2}v)\bar{\varphi}\right|\\
&\ \
+\left|\int_{\mathbb{R}^N}(I_\alpha\ast(|u|^p-|v|^p))|v|^{p-2}v\bar{\varphi}dx\right|\\
\lesssim& \left|\int_{\mathbb{R}^N}(I_\alpha\ast|u|^p)(|u|^{p-2}+|v|^{p-2})|u-v|\bar{\varphi}dx\right|\\
&\ \
+\left|\int_{\mathbb{R}^N}(I_\alpha\ast((|u|+|v|)^{p-1}|u-v|))|v|^{p-2}v\bar{\varphi}dx\right|\\
\lesssim&\|u\|_{H^1}^p(\|u\|_{H^1}^{p-2}+\|v\|_{H^1}^{p-2})\|u-v\|_{H^1}\|\varphi\|_{H^1}\\
&\ \ \ +
\|v\|_{H^1}^{p-1}(\|u\|_{H^1}^{p-1}+\|v\|_{H^1}^{p-1})\|u-v\|_{H^1}\|\varphi\|_{H^1},
\end{split}
\end{equation*}
which implies that
\begin{equation*}
\begin{split}
\|g(u)-g(v)\|_{H^{-1}}\lesssim C(M)\|u-v\|_{H^1}
\end{split}
\end{equation*}
for $\|u\|_{H^1}, \|v\|_{H^1}\leq M$. Thus, (G2) holds.

We verify (G3). Firstly, we define
$k(x,y):=\frac{\mathcal{A}}{|x-y|^{N-\alpha}}$ and for any $R>0$,
define
\begin{equation*}
k_R(x,y):=\left\{\begin{array}{ll} k(x,y),\ & k(x,y)\leq R,\\
R,\ & k(x,y)> R\\
\end{array}\right.
\end{equation*}
and
\begin{equation*}
l_R(x,y):=k(x,y)-k_R(x,y)=\left\{\begin{array}{ll} k(x,y)-R,\ & |x-y|< (\frac{\mathcal{A}}{R})^{1/(N-\alpha)},\\
0,\ & |x-y|\geq (\frac{\mathcal{A}}{R})^{1/(N-\alpha)}.
\end{array}\right.
\end{equation*}
It is easy to see that the following facts hold:

(i) $k_R(x,y)=k_R(y,x)$, $l_R(x,y)=l_R(y,x)$;

(ii) $k_R(x,y)\in L_y^{\infty}L_x^{\infty}$;

(iii) For any $y\in \mathbb{R}^N$ and $1\leq \alpha_1<N/(N-\alpha)$,
\begin{equation*}
\begin{split}
\int_{\mathbb{R}^N}|l_R(x,y)|^{\alpha_1}dx &= \int_{|x-y|<
(\frac{\mathcal{A}}{R})^{1/(N-\alpha)}}|l_R(x,y)|^{\alpha_1}dx\\
&\lesssim
\int_{|x-y|<
(\frac{\mathcal{A}}{R})^{1/(N-\alpha)}}\frac{1}{|x-y|^{(N-\alpha)\alpha_1}}dx\\
&\lesssim\int_0^{(\frac{\mathcal{A}}{R})^{1/(N-\alpha)}}
r^{-(N-\alpha)\alpha_1}r^{N-1}dr\\
&\lesssim\left(\frac{\mathcal{A}}{R}\right)^{\frac{N-(N-\alpha)\alpha_1}{N-\alpha}},
\end{split}
\end{equation*}
which implies that $l_R(x,y)\in L_y^{\infty}L_x^{\alpha_1}$ and
$l_R(x,y)\to 0$ in $L_y^{\infty}L_x^{\alpha_1}$ as $R\to \infty$ for
any $1\leq \alpha_1<N/(N-\alpha)$.

For any given $M>0$, and any $u, v\in H^1(\mathbb{R}^N)$ with
$\|u\|_{H^1}, \|v\|_{H^1}\leq M$, we calculate
\begin{equation}\label{e3.15}
\begin{split}
|G(u)-G(v)|&=\left|\int_{\mathbb{R}^N}(I_\alpha\ast|u|^p)|u|^pdx-\int_{\mathbb{R}^N}(I_\alpha\ast|v|^p)|v|^pdx\right|\\
&=\left|\int_{\mathbb{R}^N}(I_\alpha\ast(|u|^p+|v|^p))(|u|^p-|v|^p)dx\right|\\
&=\int_{\mathbb{R}^N}k_R(x,y)(|u(y)|^p+|v(y)|^p)(|u(x)|^p-|v(x)|^p)dxdy\\
&\ \qquad
+\int_{\mathbb{R}^N}l_R(x,y)(|u(y)|^p+|v(y)|^p)(|u(x)|^p-|v(x)|^p)dxdy\\
&:=I+II.
\end{split}
\end{equation}

In the following, we calculate $I$ and $II$. Since
$p<\frac{N+\alpha}{N-2}$, we can choose $\alpha_1\in
[1,N/(N-\alpha))$ such that $1/\alpha_1\leq 2(1-p/2^*)$. For such
choice of $\alpha_1$, we define
\begin{equation*}
\frac{1}{\gamma_1}:=1-\frac{1}{2}\left(\frac{1}{\alpha_1}+\frac{1}{\infty}\right),
\end{equation*}
then $2\leq p\gamma_1\leq 2^*$. For any $\delta>0$, we use Lemma
\ref{lem 4.2} and the Sobolev imbedding theorem to have
\begin{equation}\label{e3.14}
\begin{split}
II&\lesssim
\|l_R\|_{L_x^{\infty}L_y^{\alpha_1}}\||u|^p+|v|^p\|_{L^{\gamma_1}}\||u|^p-|v|^p\|_{L^{\gamma_1}}\\
&\lesssim
\|l_R\|_{L_x^{\infty}L_y^{\alpha_1}}(\|u\|_{L^{p\gamma_1}}^{2p}+\|v\|_{L^{p\gamma_1}}^{2p})\\
&\lesssim
\|l_R\|_{L_x^{\infty}L_y^{\alpha_1}}(\|u\|_{H^1}^{2p}+\|v\|_{H^1}^{2p})\\
&\leq  \delta
\end{split}
\end{equation}
by choosing $R$ large enough. For such choice of $R$, by using Lemma
\ref{lem 4.2}, the H\"{o}lder inequality, the Sobolev imbedding
theorem, (\ref{e3.12}) and  the Young inequality
\begin{equation}\label{e2.3}
\tilde{a}\tilde{b}\lesssim\epsilon\tilde{a}^{r}+\epsilon^{-q/r}\tilde{b}^{q},\
\frac{1}{r}+\frac{1}{q}=1,
\end{equation}
$\epsilon>0$ is an arbitrary real number, we obtain that
\begin{equation}\label{e3.13}
\begin{split}
I&\leq
\|k_R\|_{L_x^{\infty}L_y^{\infty}}\||u|^p+|v|^p\|_{L^1}\||u|^p-|v|^p\|_{L^1}\\
&\lesssim R\||u|^p+|v|^p\|_{L^1}\|(|u|^{p-1}+|v|^{p-1})|u-v|\|_{L^1}\\
&\lesssim R(\|u\|_{L^p}^p+\|v\|_{L^p}^p)(\|u\|_{L^p}^{p-1}+\|v\|_{L^p}^{p-1})\|u-v\|_{L^p}\\
&\lesssim
R M^{2p-1}\|u-v\|_{L^2}^{1-\theta}\|u-v\|_{L^{2^*}}^{\theta}\\
&\lesssim \epsilon
\|u-v\|_{L^{2^*}}+\epsilon^{-\theta/(1-\theta)}C(M)\|u-v\|_{L^2}\\
&\leq \delta+ C_\delta(M)\|u-v\|_{L^2},
\end{split}
\end{equation}
by choosing $\epsilon>0$ small enough, where $\theta$ is defined by
\begin{equation*}
\frac{1}{p}=\frac{1-\theta}{2}+\frac{\theta}{2^*}.
\end{equation*}
Inserting (\ref{e3.14}) and (\ref{e3.13}) into (\ref{e3.15}), we
complete the proof of (G3).

We verify (G5). Let $I\subset \mathbb{R}$ be a bounded open
interval, $\{w_n\}_{n=1}^{\infty}$ be any bounded sequence in
$L^{\infty}(I;H^1(\mathbb{R}^N))$ satisfying
\begin{equation}\label{e*1}
\begin{cases}
w_n(t)\to w(t)\ \mathrm{weakly\ in}\ H^1(\mathbb{R}^N) \ \mathrm{as
}\ n\to \infty\
\mathrm{for\ almost\ all}\ t\in I,\\
g(w_n(t))\to f(t) \ \ \mathrm{weakly *\ in}\
L^{\infty}(I;H^{-1}(\mathbb{R}^N))\ \mathrm{as}\ n\to \infty.
\end{cases}
\end{equation}
We will show that $f(t)=g(w(t))$, and thus (G5) follows in view of
(G4).

It follows from (\ref{e*1}) that $\{g(w_n)\}$ is bounded in
$L^{\infty}(I,H^{-1}(\mathbb{R}^N))$. Hence, there exists $C>0$ such
that
\begin{equation}\label{e*4}
\|g(w_n(t))\|_{H^{-1}}\leq C\ \mathrm{for\ almost\ all\ }  t\in I\
\mathrm{and\ any}\  n\in \mathbb{N}.
\end{equation}
Consequently, for almost all  $t\in I$, there exists a subsequence
$g(w_{n_j(t)}(t))$ and $\tilde{f}(t)\in H^{-1}(\mathbb{R}^N)$ such
that
\begin{equation}\label{e*2}
g(w_{n_j(t)}(t)) \to \tilde{f}(t) \ \mathrm{weakly\  in}\
H^{-1}(\mathbb{R}^N)\ \mathrm{as}\ j\to\infty.
\end{equation}
By (\ref{e*1}), (\ref{e*2}), Lemma \ref{lem3.2} and the uniqueness
of limit, we have $\tilde{f}(t)=g(w(t))$. Hence,
\begin{equation}\label{e*3}
g(w_n(t))\to g(w(t)) \ \mathrm{weakly\ in\ } H^{-1}(\mathbb{R}^N)\
\mathrm{for\ almost\ all} \ t\in I.
\end{equation}

For any $\varphi(t)\in L^{1}(I,H^{1}(\mathbb{R}^N))$, we have
$\varphi(t)\in H^{1}(\mathbb{R}^N)$ for almost all $t\in I$. Hence,
by (\ref{e*3}),
\begin{equation*}
\langle g(w_n(t)),\varphi(t)\rangle_{H^{-1},H^{1}}\to \langle
g(w(t)),\varphi(t)\rangle_{H^{-1},H^{1}}\ \mathrm{for\ almost\ all}
\ t\in I.
\end{equation*}
By (\ref{e*4}), we have
\begin{equation*}
\begin{split}
\left|\langle g(w_n(t)),\varphi(t)\rangle_{H^{-1},H^{1}}\right|&\leq
\|g(w_n(t))\|_{H^{-1}}\|\varphi(t)\|_{H^1}\\
&\leq C\|\varphi(t)\|_{H^1}\in L^1(\mathbb{R}^N).
\end{split}
\end{equation*}
Thus, the Lebesgue's dominated convergence theorem implies that
\begin{equation*}
\int_{I}\langle g(w_n(t)),\varphi(t)\rangle_{H^{-1},H^{1}}dt\to
\int_{I}\langle g(w(t)),\varphi(t)\rangle_{H^{-1},H^{1}}dt
\end{equation*}
as $n\to \infty$. That is $g(w_n(t))\to g(w(t))$ weakly $*$ in
$L^{\infty}(I,H^{-1}(\mathbb{R}^N))$. Hence, $f(t)=g(w(t))$. The
proof of (G5) is complete.
\end{proof}

\begin{lemma}\label{lem unique}
Let $N\geq 3$, $\alpha\in ((N-4)_+,N)$, $2\leq p<(N+\alpha)/(N-2)$,
$a=\pm 1$, $g$ and $G$ be as in Lemma \ref{lem condi}. Assume that
$u_j\in L^{\infty}((-T,T),H^1(\mathbb{R}^N))\cap
W^{1,\infty}((-T,T),H^{-1}(\mathbb{R}^N))$ $(j=1,2)$ are two local
weak solutions to $\mathrm{(CH_b)}$ on $(-T,T)$ with initial values
$u_j(0)=u_{j,0}$. Then for any $(p_2,q_2)\in S$,
\begin{equation*}
\|u_1(t)-u_2(t)\|_{L^{p_2}((-T,T),L^{q_2})}\leq
C\|u_{1,0}-u_{2,0}\|_{L^2},
\end{equation*}
where $C$ is a constant depending on
$\|u_1\|_{L^{\infty}((-T,T),H^1)}$ and
$\|u_2\|_{L^{\infty}((-T,T),H^1)}$.
\end{lemma}

\begin{proof}
It follows from the integral expressions of $u_j\ (j=1,2)$
\begin{equation*}
u_j(t)=e^{-it\mathcal{L}_b}u_{j,0}-i\int_{0}^{t}e^{-i(t-s)\mathcal{L}_b}g(u_j(s))ds
\end{equation*}
that $v(t):=u_1(t)-u_2(t)$  satisfies
\begin{equation}\label{e4.7}
v(t)=e^{-it\mathcal{L}_b}(u_{1,0}-u_{2,0})-i\int_{0}^{t}e^{-i(t-s)\mathcal{L}_b}(g(u_1(s))-g(u_2(s)))ds.
\end{equation}

Let $\gamma=\frac{2Np}{N+\alpha}$. By the Hardy-Littlewood-Sobolev
inequality, for any $u,\ v,\ \varphi\in H^1(\mathbb{R}^N)$,
\begin{equation*}
\begin{split}
\langle g(u)-g(v),\varphi\rangle\leq
\left(\|u\|_{L^{\frac{2Np}{N+\alpha}}}\vee\|v\|_{L^{\frac{2Np}{N+\alpha}}}\right)^{2p-2}\|u-v\|_{L^{\frac{2Np}{N+\alpha}}}
\|\varphi\|_{L^{\frac{2Np}{N+\alpha}}},
\end{split}
\end{equation*}
which implies that
\begin{equation}\label{e4.6}
\|g(u)-g(v)\|_{L^{\gamma'}}\leq
\left(\|u\|_{L^{\frac{2Np}{N+\alpha}}}\vee\|v\|_{L^{\frac{2Np}{N+\alpha}}}\right)^{2p-2}\|u-v\|_{L^{\frac{2Np}{N+\alpha}}}.
\end{equation}

Set
\begin{equation*}
(p_1,q_1)=\left(\frac{4p}{Np-(N+\alpha)},
\frac{2Np}{N+\alpha}\right)\in S.
\end{equation*}
For any $(p_2,q_2)\in S$ and $T_0\in (0,T]$, it follows from
Proposition \ref{pro 2.2}, the H\"{o}lder inequality and
(\ref{e4.6}) that
\begin{equation}\label{e4.8}
\begin{split}
&\left\|\int_{0}^{t}e^{-i(t-s)\mathcal{L}_b}(g(u_1(s))-g(u_2(s)))ds\right\|_{L^{p_2}((-T_0,T_0),L^{q_2})}\\
\lesssim & \|g(u_1(t))-g(u_2(t))\|_{L^{p_1'}((-T_0,T_0),L^{q_1'})}\\
\lesssim &
\left(\int_{-T_0}^{T_0}\left(\|u_1\|_{L^{q_1}}\vee\|u_2\|_{L^{q_1}}\right)^{(2p-2)p_1'}\|v\|_{L^{q_1}}^{p_1'}dt\right)^{1/p_1'}\\
\lesssim& \left(\int_{-T_0}^{T_0}
\left(\|u_1\|_{L^{q_1}}\vee\|u_2\|_{L^{q_1}}\right)
^{\frac{4p(p-1)}{(2-N)p+(N+\alpha)}}dt\right)^{\frac{(2-N)p+(N+\alpha)}{2p}}\\
&\qquad\qquad\qquad\times\|v\|_{L^{p_1}((-T_0,T_0),L^{q_1})}\\
\lesssim
&\left(\|u_1\|_{L^{\infty}((-T_0,T_0),L^{q_1})}+\|u_2\|_{L^{\infty}((-T_0,T_0),L^{q_1})}\right)^{2(p-1)}\\
&\qquad\qquad \qquad\times
T_0^{\frac{(2-N)p+(N+\alpha)}{2p}}\|v\|_{L^{p_1}((-T_0,T_0),L^{q_1})}.
\end{split}
\end{equation}

By (\ref{e4.7}), (\ref{e4.8}) and Proposition \ref{pro 2.2}, we
obtain that
\begin{equation}\label{e4.5}
\begin{split}
\|v\|_{L^{p_2}((-T_0,T_0),L^{q_2})}&\leq
C\|u_{1,0}-u_{2,0}\|_{L^2}\\
&\qquad
+CM^{2(p-1)}T_0^{\frac{(2-N)p+(N+\alpha)}{2p}}\|v\|_{L^{p_1}((-T_0,T_0),L^{q_1})},
\end{split}
\end{equation}
where
$M=\|u_1\|_{L^{\infty}((-T,T),H^1)}\vee\|u_2\|_{L^{\infty}((-T,T),H^1)}$.

Since $p<\frac{N+\alpha}{N-2}$, by choosing $(p_2,q_2)=(p_1,q_1)$
and $T_0\in (0,T]$ such that
$$CM^{2(p-1)}T_0^{\frac{(2-N)p+(N+\alpha)}{2p}}<1/2,$$
 we obtain from (\ref{e4.5}) that
\begin{equation*}
\|v\|_{L^{p_1}((-T_0,T_0),L^{q_1})}\leq 2C\|u_{1,0}-u_{2,0}\|_{L^2},
\end{equation*}
and from (\ref{e4.5}) again that
\begin{equation*}
\|v\|_{L^{p_2}((-T_0,T_0),L^{q_2})}\leq 2C\|u_{1,0}-u_{2,0}\|_{L^2}
\end{equation*}
for any $(p_2,q_2)\in S$.

Extending the interval for finite steps, similarly to the above
arguments, we obtain that
\begin{equation*}
\|v\|_{L^{p_1}((-T,T),L^{q_1})}\leq C\|u_{1,0}-u_{2,0}\|_{L^2}
\end{equation*}
and
\begin{equation*}
\|v\|_{L^{p_2}((-T,T),L^{q_2})}\leq C\|u_{1,0}-u_{2,0}\|_{L^2}
\end{equation*}
for any $(p_2,q_2)\in S$. The proof is complete.
\end{proof}

\textbf{Proof of Theorem \ref{thm local}}.  It is a direct result of
Theorem \ref{thm4.1}, Lemmas \ref{lem condi} and  \ref{lem unique}.

\section{Sharp Gagliardo-Nirenberg inequality and global existence}

\setcounter{section}{4} \setcounter{equation}{0}

In this section, we derive the sharp Gagliardo-Nirenberg inequality
and some useful identities  with respect to problem
($\mathrm{CH_b}$). Based of which and the local well-posedness
result established in Section 3, we prove the global existence
results in various cases.

\textbf{4.1.  Sharp Gagliardo-Nirenberg inequality}.  We consider
the sharp Gagliardo-Nirenberg inequality
\begin{equation}\label{e3.1}
\int_{\mathbb{R}^N}(I_\alpha\ast |u|^p)|u|^pdx\leq
C_{GN}(b)\|u\|_{L^2}^{N+\alpha-Np+2p}\|u\|_{\dot{H}_b^1}^{Np-N-\alpha},
\end{equation}
where the sharp constant $C_{GN}(b)$ is defined by
\begin{equation}\label{e3.2}
C^{-1}_{GN}(b)=\inf\{J_b(u):u\in H_b^1(\mathbb{R}^N)\setminus\{0\}\}
\end{equation}
and
\begin{equation}\label{e3.3}
J_b(u)=\frac{\|u\|_{L^2}^{N+\alpha-Np+2p}\|u\|_{\dot{H}_b^1}^{Np-N-\alpha}}{\int_{\mathbb{R}^N}(I_\alpha\ast
|u|^p)|u|^pdx}.
\end{equation}

We also consider the sharp radial Gagliardo-Nirenberg inequality
\begin{equation}\label{e3.4}
\int_{\mathbb{R}^N}(I_\alpha\ast |u|^p)|u|^pdx\leq
C_{GN}(b,\mathrm{rad})\|u\|_{L^2}^{N+\alpha-Np+2p}\|u\|_{\dot{H}_b^1}^{Np-N-\alpha},\
u\ \mathrm{is\ radial},
\end{equation}
where the sharp constant $C_{GN}(b, \mathrm{rad})$ is defined by
\begin{equation}
C^{-1}_{GN}(b,\mathrm{rad})=\inf\{J_b(u):u\in
H_b^1(\mathbb{R}^N)\setminus\{0\},\ u \ \mathrm{is\ radial}\}.
\end{equation}

By combining the proofs of Theorem 4.1 in \cite{Dinh 2017} and
Theorem 2.3 in \cite{Feng-Yuan 2015}, we obtain the following
result. For the completeness, we give the proof here.

\begin{lemma}\label{lem3.1}
Let $N\geq 3$, $\alpha\in (0,N)$, $(N+\alpha)/N<p<(N+\alpha)/(N-2)$,
$b>-(N-2)^2/4$. Then $C_{GN}(b)\in (0,\infty)$ and

(1) If $-(N-2)^2/4<b\leq 0$, then the equality in (\ref{e3.1}) is
attained by a nontrivial  radial  ground state $Q_b\in
H_b^1(\mathbb{R}^N)$ to equation (\ref{e3.5}). Moreover,
\begin{equation}\label{e3.6}
C_{GN}(b)=\frac{2p}{N+\alpha-(N-2)p}\left(\frac{N+\alpha-(N-2)p}{Np-N-\alpha}\right)^{\frac{Np-N-\alpha}{2}}\|Q_b\|_{L^2}^{2-2p},
\end{equation}
\begin{equation}\label{e3.7}
\int_{\mathbb{R}^N}(I_\alpha\ast|Q_b|^p)|Q_b|^{p}dx=\frac{2p}{N+\alpha-(N-2)p}\|Q_b\|_{L^2}^2,
\end{equation}
and
\begin{equation}\label{e3.8}
\|Q_b\|_{\dot{H}_b^1}^2=\frac{Np-N-\alpha}{N+\alpha-(N-2)p}\|Q_b\|_{L^2}^2.
\end{equation}

(2) If $b>0$, then $C_{GN}(b)=C_{GN}(0)$ and the equality in
(\ref{e3.1}) is never attained. However, $C_{GN}(b,\mathrm{rad})$ is
attained by a radial solution $Q_{b,\mathrm{rad}}$ to equation
(\ref{e3.5}). Moreover, the same identities as in
(\ref{e3.6})-(\ref{e3.8}) hold true with $Q_{b,\mathrm{rad}}$,
$C_{GN}(b,\mathrm{rad})$ in place of $Q_{b}$, $C_{GN}(b)$
respectively.
\end{lemma}

\begin{proof}
The case $b=0$ is already considered in \cite{Feng-Yuan 2015}, so we
assume that $b\neq0$. By using the Hardy-Littlewood-Sobolev
inequality, the interpolation inequality, the Sobolev imbedding
theorem and the equivalence of $\|\cdot\|_{\dot{H}_b^1}$ and
$\|\cdot\|_{\dot{H}^1}$,  for any $u\in H_b^1(\mathbb{R}^N)$, we
obtain that
\begin{equation*}
\begin{split}
\int_{\mathbb{R}^N}(I_\alpha\ast|u|^p)|u|^{p}dx&\leq
C\|u\|_{L^{\frac{2Np}{N+\alpha}}}^{2p}\\
&\leq C\|u\|_{\dot{H}^1}^{Np-N-\alpha}\|u\|_{L^2}^{N+\alpha-Np+2p}\\
&\leq C
\|u\|_{\dot{H}_b^1}^{Np-N-\alpha}\|u\|_{L^2}^{N+\alpha-Np+2p},
\end{split}
\end{equation*}
which implies that $0<C^{-1}_{GN}(b)<\infty$.

\textbf{Case 1} ($-\frac{(N-2)^2}{4}<b<0$). Let
$\{u_n\}_{n=1}^{\infty}\subset H_b^1(\mathbb{R}^N)\setminus\{0\}$ be
a minimizing sequence of $J_b$, that is, $\lim_{n\to
\infty}J_b(u_n)=C^{-1}_{GN}(b)$. Denote by $u_n^*$ the Schwartz
symmetrization of $u_n$. From Chapter 3 in \cite{Lieb-Loss 2001}, we
have
\begin{equation*}
\int_{\mathbb{R}^N}(I_\alpha\ast|u_n|^p)|u_n|^{p}dx\leq
\int_{\mathbb{R}^N}(I_\alpha\ast|u_n^*|^p)|u_n^*|^{p}dx,
\end{equation*}
\begin{equation*}
\int_{\mathbb{R}^N}|u_n|^{2}dx= \int_{\mathbb{R}^N}|u_n^*|^{2}dx,
\end{equation*}
\begin{equation*}
\int_{\mathbb{R}^N}|\nabla u_n|^{2}dx\geq \int_{\mathbb{R}^N}|\nabla
u_n^*|^{2}dx,
\end{equation*}
and
\begin{equation*}
\int_{\mathbb{R}^N}|x|^{-2}|u_n|^{2}dx\leq
\int_{\mathbb{R}^N}|x|^{-2}|u_n^*|^{2}dx.
\end{equation*}
The above inequalities and $b<0$ imply that $J_b(u_n^*)\leq
J_b(u_n)$. Hence, we may assume that $\{u_n\}_{n=1}^{\infty}$ is
radial.

Direct calculation shows that $J_b$ is invariant under the scaling
\begin{equation*} u_{\lambda,\mu}(x)=\lambda u(\mu x),\ \lambda,\
\mu>0,
\end{equation*}
that is, $J_b(u_{\lambda,\mu})=J_b(u)$. For the sake, we set
$v_n=\lambda_nu_n(\mu_nx)$ with
\begin{equation*}
\lambda_n=\frac{\|u_n\|_{L^2}^{N/2-1}}{\|u_n\|_{\dot{H}_b^1}^{N/2}}\
\ \mathrm{and}\ \  \mu_n=
\frac{\|u_n\|_{L^2}}{\|u_n\|_{\dot{H}_b^1}}.
\end{equation*}
Then $\|v_n\|_{L^2}=\|v_n\|_{\dot{H}_b^1}=1$. That is,
$\{v_n\}_{n=1}^{\infty}\subset H_b^1(\mathbb{R}^N)\setminus\{0\}$ is
a bounded radial symmetric minimizing sequence of $J_b$. Hence,
there exists $v\in H_b^1(\mathbb{R}^N)$ such that
\begin{equation*}
v_n\to v\ \ \mathrm{weakly\ in\ }H_b^1(\mathbb{R}^N),
\end{equation*}
\begin{equation*}
v_n\to v\ \ \mathrm{strongly\ in\ }L^r(\mathbb{R}^N)\ \mathrm{with}\
r\in\left(2,2^*\right),
\end{equation*}
which combine with the Hardy-Littlewood-Sobolev inequality and the
lower semi-continuity of the norm imply that
\begin{equation*}
\|v\|_{L^2}\leq \liminf_{n\to \infty}\|v_n\|_{L^2},\quad
\|v\|_{\dot{H}_b^1}\leq \liminf_{n\to \infty}\|v_n\|_{\dot{H}_b^1},
\end{equation*}
\begin{equation*}
\int_{\mathbb{R}^N}(I_\alpha\ast|v|^p)|v|^{p}dx=\lim_{n\to
\infty}\int_{\mathbb{R}^N}(I_\alpha\ast|v_n|^p)|v_n|^{p}dx.
\end{equation*}
Hence, $$J_b(v)\leq \liminf_{n\to\infty}J_b(v_n)=C^{-1}_{GN}(b).$$
By the definition of $C^{-1}_{GN}(b)$, we obtain that
\begin{equation*}
J_b(v)=\lim_{n\to\infty}J_b(v_n)=\inf_{u\in
H_b^1(\mathbb{R}^N)\setminus\{0\}} J_b(u)=C^{-1}_{GN}(b)
\end{equation*}
and $\|v\|_{L^2}=\|v\|_{\dot{H}_b^1}=1$. In particular, $v$
satisfies
\begin{equation*}
\frac{d}{d\epsilon}J_b(v+\epsilon\varphi)\mid_{\epsilon=0}=0\
\mathrm{for\ any\ } \varphi\in H_b^1(\mathbb{R}^N).
\end{equation*}
Consequently, $v$ satisfies the elliptic equation
\begin{equation*}
\frac{C_{GN}(b)(Np-N-\alpha)}{2p}\mathcal{L}_bv+\frac{C_{GN}(b)(N+\alpha-Np+2p)}{2p}v=(I_\alpha\ast|v|^p)|v|^{p-2}v.
\end{equation*}
Set $v(x)=\lambda Q_b(\mu x)$ with
\begin{equation*}
\lambda=\left(\frac{(N+\alpha-Np+2p)^{\alpha/2+1}C_{GN}(b)}{2p(Np-N-\alpha)^{\alpha/2}}\right)^{1/(2p-2)}
\end{equation*}
and
\begin{equation*}
\mu=\left(\frac{N+\alpha-Np+2p}{Np-N-\alpha}\right)^{1/2},
\end{equation*}
then $Q_b(x)$ satisfies (\ref{e3.5}).

Multiplying (\ref{e3.5}) by $Q_b$ and $x\cdot \nabla Q_b$ and
integrating on $\mathbb{R}^N$, respectively, we obtain that
\begin{equation*}
\|Q_b\|_{\dot{H}_b^1}^2+\|Q_b\|_{L^2}^2=\int_{\mathbb{R}^N}(I_\alpha\ast|Q_b|^p)|Q_b|^{p}dx
\end{equation*}
and
\begin{equation*}
(N-2)\|Q_b\|_{\dot{H}_b^1}^2+N\|Q_b\|_{L^2}^2=\frac{N+\alpha}{p}\int_{\mathbb{R}^N}(I_\alpha\ast|Q_b|^p)|Q_b|^{p}dx,
\end{equation*}
which  implies (\ref{e3.7}) and (\ref{e3.8}). Since
$C^{-1}_{GN}(b)=J_b(Q_b)$, we obtain (\ref{e3.6}). The functional of
(\ref{e3.5}) is defined by
\begin{equation*}
\begin{split}
\hat{E}(Q)&=\frac{1}{2}\|Q\|_{\dot{H}_b^1}^2+\frac{1}{2}\|Q\|_{L^2}^2-\frac{1}{2p}\int_{\mathbb{R}^N}(I_\alpha\ast|Q|^p)|Q|^pdx\\
&=\frac{p-1}{N+\alpha-(N-2)p}\|Q\|_{L^2}^2,
\end{split}
\end{equation*}
which implies that $Q_b$ is a ground state of (\ref{e3.5}).

\textbf{Case 2} ($b>0$). Choose  a sequence
$\{x_n\}_{n=1}^{\infty}\subset \mathbb{R}^N$ with $|x_n|\to \infty$.
Let $Q_0$ be a positive radial ground state to
\begin{equation*}
\mathcal{L}_0Q+Q=(I_\alpha\ast|Q|^p)|Q|^{p-2}Q.
\end{equation*}
Then by Lemma \ref{lem 2.3},
\begin{equation*}
J_b(Q_0(\cdot-x_n))\to J_0(Q_0)=C^{-1}_{GN}(0),
\end{equation*}
which implies that $C^{-1}_{GN}(b)\leq C^{-1}_{GN}(0)$.

On the other hand, since $b>0$,
$\|u\|_{\dot{H}^1}<\|u\|_{\dot{H}_b^1}$ for any $u\in
H^1(\mathbb{R}^N)\setminus\{0\}$. The sharp Gagliardo-Nirenberg
inequality for $b=0$ implies that
\begin{equation}\label{e3.9}
\begin{split}
\int_{\mathbb{R}^N}(I_\alpha\ast |u|^p)|u|^pdx&\leq
C_{GN}(0)\|u\|_{L^2}^{N+\alpha-Np+2p}\|u\|_{\dot{H}^1}^{Np-N-\alpha}\\
&<C_{GN}(0)\|u\|_{L^2}^{N+\alpha-Np+2p}\|u\|_{\dot{H}_b^1}^{Np-N-\alpha}.
\end{split}
\end{equation}
Hence, $J_b(u)>G^{-1}_{GN}(0)$ for any $u\in
H^1(\mathbb{R}^N)\setminus\{0\}$. Since $H^1(\mathbb{R}^N)$ is
equivalent to $H_b^1(\mathbb{R}^N)$, we obtain that
$C^{-1}_{GN}(b)\geq C^{-1}_{GN}(0)$. Thus, $C_{GN}(b)= C_{GN}(0)$.
The inequality (\ref{e3.9}) also implies that the equality in
(\ref{e3.1}) is never attained.

If we only consider radial functions, the result follows exactly as
 case 1. The proof is complete.
\end{proof}

\begin{remark}\label{rek4.1}
(1). When $-\frac{(N-2)^2}{4}<b\leq 0$, from Lemma \ref{lem3.1} and
the definitions of $H(b)$ and $K(b)$, we obtain that
\begin{equation}\label{e4.21}
K(b)=\left(\frac{2p}{(Np-N-\alpha)C_{GN}(b)}\right)^{1/(Np-N-\alpha-2)}
\end{equation}
and
\begin{equation}\label{e4.22}
H(b)=\frac{Np-N-\alpha-2}{2(Np-N-\alpha)}\left(\frac{2p}{(Np-N-\alpha)C_{GN}(b)}\right)^{2/(Np-N-\alpha-2)}.
\end{equation}

(2). Similarly, when $b>0$, the same identities as in (\ref{e4.21})
and (\ref{e4.22}) hold true with $K(b,\mathrm{rad})$,
$H(b,\mathrm{rad})$, $C_{GN}(b,\mathrm{rad})$ in place of $K(b)$,
$H(b)$, $C_{GN}(b)$ respectively.
\end{remark}

\textbf{4.2.   Global existence.}

\textbf{Proof of the global existence parts in Theorems \ref{th1.2},
\ref{thm1.3} and \ref{thm1.4}}.  Let  $u\in
C(I,H^1(\mathbb{R}^N))\cap C^1(I,H^{-1}(\mathbb{R}^N))$ be the weak
solution obtained in Theorem \ref{thm local} with initial value
$u_0(x)\in H^1(\mathbb{R}^N)$. In view of the conservation laws, we
just need to bound $\|u(t)\|_{\dot{H}_b^1}$ for any $t\in I$. It is
trivial for the defocusing case $a=-1$ from the conservation laws.
In the following, we consider the focusing case $a=1$.

\textbf{Case 1} ($2\leq p<1+(2+\alpha)/N$). By using the sharp
Gagliardo-Nirenberg inequality and the conservation laws, we obtain
that
\begin{equation*}
\begin{split}
\|u(t)\|_{\dot{H}_b^1}^2&=\int_{\mathbb{R}^N}(|\nabla
u(t)|^2+b|x|^{-2}|u(t)|^2)dx\\
&=2 E_b(u(t))+\frac{1}{p}\int_{
\mathbb{R}^N}(I_\alpha\ast|u|^p)|u|^pdx\\
&\leq
2E_b(u(t))+\frac{1}{p}C_{GN}(b)\|u(t)\|_{L^2}^{N+\alpha-Np+2p}\|u(t)\|_{\dot{H}_b^1}^{Np-N-\alpha}\\
&=2E_b(u_0)+\frac{1}{p}C_{GN}(b)\|u_0\|_{L^2}^{N+\alpha-Np+2p}\|u(t)\|_{\dot{H}_b^1}^{Np-N-\alpha}.
\end{split}
\end{equation*}
Since $p<1+(2+\alpha)/N$, we have $Np-N-\alpha<2$ and thus, by using
the Young inequality, $\|u(t)\|_{\dot{H}_b^1}$ is bounded for any
$t\in I$. This proves the global existence result in Theorem
\ref{th1.2}.

\textbf{Case 2} ($p=1+(2+\alpha)/N$). In this case, we have
\begin{equation*}
N+\alpha-Np+2p=2(2+\alpha)/N,\ Np-N-\alpha=2,
\end{equation*}
\begin{equation*}
2-2p=-2(2+\alpha)/N,\ C_{GN}(b)=p\|Q_b\|_{L^2}^{2-2p}.
\end{equation*}
Similarly to  case 1, we obtain that
\begin{equation*}
\begin{split}
\|u(t)\|_{\dot{H}_b^1}^2&=2 E_b(u(t))+\frac{1}{p}\int_{
\mathbb{R}^N}(I_\alpha\ast|u|^p)|u|^pdx\\
&\leq
2E_b(u(t))+\frac{1}{p}C_{GN}(b)\|u(t)\|_{L^2}^{N+\alpha-Np+2p}\|u(t)\|_{\dot{H}_b^1}^{Np-N-\alpha}\\
&=2E_b(u_0)+\left(\frac{\|u_0\|_{L^2}}{\|Q_{b\wedge
0}\|_{L^2}}\right)^{2(2+\alpha)/N}\|u(t)\|_{\dot{H}_b^1}^{2}.
\end{split}
\end{equation*}
Since $\|u_0\|_{L^2}<\|Q_{b\wedge 0}\|_{L^2}$, we obtain that
$\|u(t)\|_{\dot{H}_b^1}$  is bounded for any $t\in I$. This proves
the global existence result (i) in Theorem \ref{thm1.3}.

\textbf{Case 3} ($1+(2+\alpha)/N<p<(N+\alpha)/(N-2)$). Multiplying
both sides of $E_b(u(t))$ by $\|u(t)\|_{L^2}^{2\sigma}$ with
$\sigma=\frac{N+\alpha-Np+2p}{Np-N-\alpha-2}$ and by using the sharp
Gagliardo-Nirenberg inequality, we obtain that
\begin{equation}\label{e6.1}
\begin{split}
&E_b(u(t))\|u(t)\|_{L^2}^{2\sigma}\\
&=\left(\frac{1}{2}\|u\|_{\dot{H}_b^1}^2
-\frac{1}{2p}\int_{\mathbb{R}^N}(I_\alpha\ast|u|^p)|u|^pdx\right)\|u(t)\|_{L^2}^{2\sigma}\\
&\geq \frac{1}{2}(\|u\|_{\dot{H}_b^1}\|u(t)\|_{L^2}^{\sigma})^2
-\frac{1}{2p}C_{GN}(b)\|u(t)\|_{L^2}^{N+\alpha-Np+2p+2\sigma}\|u\|_{\dot{H}_b^1}^{Np-N-\alpha}\\
&=\frac{1}{2}(\|u\|_{\dot{H}_b^1}\|u(t)\|_{L^2}^{\sigma})^2-\frac{C_{GN}(b)}{2p}(\|u\|_{\dot{H}_b^1}\|u(t)\|_{L^2}^{\sigma})^{Np-N-\alpha}\\
&=f(\|u\|_{\dot{H}_b^1}\|u(t)\|_{L^2}^{\sigma}),
\end{split}
\end{equation}
where
\begin{equation*}
f(s):=\frac{1}{2}s^2-\frac{C_{GN}(b)}{2p} s^{Np-N-\alpha}, \ s\in
[0,\infty).
\end{equation*}

Direct calculation shows that $f$ has a unique critical point $s^*$
which corresponds to its maximum point, where
\begin{equation*}
s^*=\left(\frac{2p}{(Np-N-\alpha)C_{GN}(b)}\right)^{1/(Np-N-\alpha-2)}
\end{equation*}
and
\begin{equation*}
f(s^*)=\frac{Np-N-\alpha-2}{2(Np-N-\alpha)}\left(\frac{2p}{(Np-N-\alpha)C_{GN}(b)}\right)^{2/(Np-N-\alpha-2)}.
\end{equation*}
From Remark \ref{rek4.1}, we know  $K(b)=s^*$ and
$H(b)=f(s^*)=f(K(b))$.

By (\ref{e6.1}), the conservation laws, and the assumptions in
Theorem \ref{thm1.4}, we know
\begin{equation*}
\begin{split}
f(\|u(t)\|_{\dot{H}_b^1}\|u(t)\|_{L^2}^{\sigma})&\leq
E_b(u(t))\|u(t)\|_{L^2}^{2\sigma}\\
&=E_b(u_0)\|u_0\|_{L^2}^{2\sigma}\\
&<H(b)
\end{split}
\end{equation*}
for any $t\in I$. Since
$\|u_0\|_{\dot{H}_b^1}\|u_0\|_{L^2}^{\sigma}<K(b)$, by the
continuity argument, we have
$\|u(t)\|_{\dot{H}_b^1}\|u(t)\|_{L^2}^{\sigma}<K(b)$ for any $t\in
I$. This proves the global existence result (i) in Theorem
\ref{thm1.4}.

\section{Virial identities and blowup}

\setcounter{section}{5} \setcounter{equation}{0}

In this section, we first establish the virial identities for the
solution to $\mathrm{(CH_b)}$, and then, based of which, we prove
the blowup results in the main theorems. So throughout this section,
we assume $a=1$ in $\mathrm{(CH_b)}$.

\textbf{5.1.  Virial identities}.

For any $\delta>0$, we consider the following approximation problem
 of ($\mathrm{CH_b}$):
\begin{equation*}\mathrm{(CH_b^\delta)}\qquad
\begin{cases}
i\partial_tu+\Delta
u=\frac{b}{|x|^2+\delta}u-(I_{\alpha}\ast|u|^{p})|u|^{p-2}u:=-g_\delta(u),
\quad (t,x)\in
\mathbb{R}\times\mathbb{R}^{N},\\
u(0,x)=u_0(x),\quad x\in \mathbb{R}^{N}.
\end{cases}
\end{equation*}
When $\delta=0$, we refer $\mathrm{(CH_b^\delta)}$ to
$\mathrm{(CH_b)}$. It follows from (\ref{e1.1}) that for any
$b>-\frac{(N-2)^2}{4}$, $\delta>0$ and $\varphi\in
H^1(\mathbb{R}^N)$, the inequality
\begin{equation}\label{e5.13}
\begin{split}
\left(1-\frac{4b_-}{(N-2)^2}\right)\|\nabla
\varphi\|_{L^2}^2&\leq\|\nabla\varphi\|_{L^2}^2+\int_{\mathbb{R}^N}\frac{b|\varphi|^2}{|x|^2+\delta}dx\\
&\leq \left(1+\frac{4b_+}{(N-2)^2}\right)\|\nabla \varphi\|_{L^2}^2
\end{split}
\end{equation}
holds.  Moreover, by the Lebesgue dominated convergence theorem,
\begin{equation}\label{e5.21}
\left(-\Delta +\frac{b}{|x|^2+\delta}\right)\varphi\to \left(-\Delta
+\frac{b}{|x|^2}\right)\varphi\ \mathrm{strongly\ in\
}H^{-1}(\mathbb{R}^N)\ \mathrm{as}\ \delta\to 0.
\end{equation}
Hence, $\mathrm{(CH_b^\delta)}$ is a good approximation of
$\mathrm{(CH_b)}$.

For the Cauchy problem $\mathrm{(CH_b^\delta)}$, we have the
following facts:

\begin{proposition}\label{pro5.1}
Let $N\geq 3$, $\alpha\in((N-4)_+,N)$,
$p\in[2,\frac{N+\alpha}{N-2})$, $b>-\frac{(N-2)^2}{4}$. Then for any
$\delta>0$ and $u_0\in H^1(\mathbb{R}^N)$, there exists a unique
maximal-lifespan  solution $u\in C^1(I, H^1(\mathbb{R}^N))\cap C(I,
H^{-1}(\mathbb{R}^N))$ to $\mathrm{(CH_b^\delta)}$, and the
conservation laws hold:
\begin{equation*}
\|u(t)\|_{L^2}=\|u_0\|_{L^2},\
E_{b,\delta}(u(t))=E_{b,\delta}(u_0),\ \mathrm{for\ any\ }t\in I,
\end{equation*}
where
\begin{equation*}
\begin{split}
E_{b,\delta}(u(t))&=\frac{1}{2}\int_{\mathbb{R}^N}(|\nabla
u(t,x)|^2+\frac{b}{|x|^2+\delta}|u(t,x)|^2)dx\\
&\qquad-\frac{1}{2p}\int_{\mathbb{R}^N}(I_\alpha\ast|u(t,x)|^p)|u(t,x)|^{p}
dx.
\end{split}
\end{equation*}
See Theorem 4.3.1 in \cite{Cazenave 2003}. Moreover, the solution
depends continuously on the initial value, that is, if $u_{n0}\to
u_0$ strongly in $H^1(\mathbb{R}^N)$, then $u_n\to u$ in
$C([-T_1,T_2], H^1(\mathbb{R}^N))$  as $n\to \infty$ for any
$[-T_1,T_2]\subset I$. See Theorem 3.3.9 in \cite{Cazenave 2003}.
Moreover, if $u_0\in H^{2}(\mathbb{R}^N)$, then $u(t)\in C(I,
H^2(\mathbb{R}^N))\cap C^1(I, L^2(\mathbb{R}^N))$, see Theorem 4.8.1
in \cite{Cazenave 2003}.
\end{proposition}

In the following, for any $\delta>0$, we obtain the virial identity
for the solution $u_\delta$ to $\mathrm{(CH_b^\delta)}$, and then by
letting $\delta\to 0$, we obtain the virial identity for the
solution $u$ to $\mathrm{(CH_b)}$. Let $u$ be the solution to
($\mathrm{CH_b^\delta}$) with $\delta\geq 0$ and $w(x)=|x|^2$ or
$w(x)=\varphi_R(x)$, where $\varphi_R(x)$ is defined in
(\ref{e5.61}). We define
\begin{equation*}
V_w(t):=\int_{\mathbb{R}^N}w(x)|u(t,x)|^2dx.
\end{equation*}

Then we have the following virial identity.

\begin{lemma}\label{lem5.1}
Let $N\geq 3$, $\alpha\in ((N-4)_+,N)$, $p\in
[2,\frac{N+\alpha}{N-2})$, $b>-\frac{(N-2)^2}{4}$, $u_0\in
H^1(\mathbb{R}^N)$ with $\sqrt{w(x)}u_0\in L^2(\mathbb{R}^N)$,
$\delta>0$. Assume $u(t)\in C(I,H^1(\mathbb{R}^N))\cap
C^1(I,H^{-1}(\mathbb{R}^N))$ is the maximal-lifespan solution to
$\mathrm{(CH_b^\delta)}$ with initial value $u_0$, then $
\sqrt{w(x)}u(t,x)\in C(I, L^2(\mathbb{R}^N))$. Moreover, $V_w(t)\in
C^2(I,\mathbb{R})$,
\begin{equation}\label{e5.1}
\frac{d}{dt}V_w(t)=2\mathrm{Im}\int_{\mathbb{R}^N}\bar{u}\nabla
w\cdot\nabla udx
\end{equation}
and
\begin{equation}\label{e5.2}
\begin{split}
\frac{d^2}{dt^2}V_w(t)&=4\mathrm{Re}\int_{\mathbb{R}^N}\partial_{i}u\partial_j\bar{u}\partial_{ij}
wdx-\int_{\mathbb{R}^N}|u|^2\Delta^2wdx+4b\int_{\mathbb{R}^N}\frac{x\cdot\nabla
w}{(|x|^2+\delta)^2}|u|^2dx\\
&\
-\frac{2\mathcal{A}(N-\alpha)}{p}\int_{\mathbb{R}^N}\int_{\mathbb{R}^N}\frac{(x-y)\cdot(\nabla
w(x)-\nabla w(y))|u(x)|^p|u(y)|^p}{|x-y|^{N-\alpha+2}}dxdy\\
&\
+\left(\frac{4}{p}-2\right)\int_{\mathbb{R}^N}(I_\alpha\ast|u|^p)|u|^p\Delta
wdx
\end{split}
\end{equation}
for any $t\in I$.
\end{lemma}

\begin{proof}
We follow the proof of Proposition 6.5.1 in \cite{Cazenave 2003}. By
Lemma 6.5.2 in \cite{Cazenave 2003} and Proposition \ref{pro5.1}, we
obtain that $\sqrt{w(x)}u(t,x)\in C(I, L^2(\mathbb{R}^N))$,
$V_w(t)\in C^1(I,\mathbb{R})$ and (\ref{e5.1}) holds for all $t\in
I$. It remains to show that $V_w(t)\in C^2(I,\mathbb{R})$ and
(\ref{e5.2}) holds. The proof we give below is based on two
regularizations. Therefore, we proceed in two steps.

Step 1. The case $u_0\in H^2(\mathbb{R}^N)$.  By the $H^2$
regularity (Proposition \ref{pro5.1}), $u(t)\in C(I,
H^2(\mathbb{R}^N))\cap C^1(I, L^2(\mathbb{R}^N))$. For any
$\epsilon>0$, we define $\theta_\epsilon(x)=e^{-\epsilon w(x)}$ and
\begin{equation*}
h_\epsilon(t)=\mathrm{Im}\int_{\mathbb{R}^N}\theta_\epsilon(x)\bar{u}\nabla
w\cdot \nabla udx,\ \mathrm{for\ any\ }t\in I.
\end{equation*}
Then $h_\epsilon\in C^1(I,\mathbb{R})$ and
\begin{equation}\label{e5.3}
\begin{split}
h_\epsilon'(t)&=\mathrm{Re}\int_{\mathbb{R}^N}iu_t(2\theta_\epsilon\nabla\bar{u}\cdot\nabla
w+\theta_{\epsilon}\bar{u}\Delta
w+\bar{u}\nabla\theta_{\epsilon}\cdot \nabla w)dx\\
&=-\mathrm{Re}\int_{\mathbb{R}^N}(\Delta
u+g_\delta(u))(2\theta_\epsilon\nabla\bar{u}\cdot\nabla
w+\theta_{\epsilon}\bar{u}\Delta
w+\bar{u}\nabla\theta_{\epsilon}\cdot \nabla w)dx.
\end{split}
\end{equation}
Here, we have used the Green's formula and equation
$\mathrm{(CH_b^\delta)}$.

For $u\in H^2(\mathbb{R}^N)$, by using the Green's formula and
elementary calculation, we know
\begin{equation}\label{e5.4}
\begin{split}
\mathrm{Re}\int_{\mathbb{R}^N}\Delta u
2\theta_\epsilon\nabla\bar{u}\cdot\nabla w
dx=-\mathrm{Re}\int_{\mathbb{R}^N}&\left(2\nabla
u\cdot\nabla\theta_{\epsilon}\nabla\bar{u}\cdot\nabla w+2\theta_{\epsilon}\partial_ju\partial_i\bar{u}\partial_{ij}w\right.\\
&\quad\left.-|\nabla u|^2 \nabla\theta_{\epsilon}\cdot\nabla
w-|\nabla u|^2 \theta_{\epsilon}\Delta w\right)dx,
\end{split}
\end{equation}
\begin{equation}\label{e5.5}
\begin{split}
\mathrm{Re}\int_{\mathbb{R}^N}\Delta u\theta_{\epsilon}\bar{u}\Delta
wdx=-\mathrm{Re}\int_{\mathbb{R}^N}&\left(\nabla
u\cdot\nabla\theta_{\epsilon}\bar{u}\Delta
w+\theta_{\epsilon}|\nabla u|^2\Delta w\right.\\
&\ \left.-\frac{1}{2}|u|^2\nabla\theta_{\epsilon}\cdot\nabla(\Delta
w)-\frac{1}{2}|u|^2\theta_{\epsilon}\Delta^2w\right)dx,
\end{split}
\end{equation}
\begin{equation}\label{e5.6}
\begin{split}
\mathrm{Re}\int_{\mathbb{R}^N}&\Delta
u\bar{u}\nabla\theta_{\epsilon}\cdot
\nabla wdx\\
&=-\mathrm{Re}\int_{\mathbb{R}^N}\left(|\nabla u|^2
\nabla\theta_{\epsilon}\cdot\nabla
w+\bar{u}\partial_ju\partial_{ij}\theta_{\epsilon}\partial_iw
+\bar{u}\partial_ju\partial_{i}\theta_{\epsilon}\partial_{ij}w\right)dx,
\end{split}
\end{equation}
\begin{equation}\label{e5.7}
\begin{split}
\mathrm{Re}\int_{\mathbb{R}^N}&\frac{b}{|x|^2+\delta}u2\theta_{\epsilon}\nabla\bar{u}\cdot\nabla
wdx\\
&=\int_{\mathbb{R}^N}\frac{b}{|x|^2+\delta}\theta_{\epsilon}\nabla(|u|^2)\cdot\nabla
wdx\\
&=\int_{\mathbb{R}^N}\left(\frac{2b\theta_{\epsilon}x\cdot\nabla
w}{(|x|^2+\delta)^2}|u|^2-\frac{b\nabla \theta_{\epsilon}\cdot\nabla
w}{|x|^2+\delta}|u|^2-\frac{b\theta_{\epsilon}\Delta
w}{|x|^2+\delta}|u|^2\right)dx
\end{split}
\end{equation}
and
\begin{equation}\label{e5.8}
\begin{split}
\mathrm{Re}\int_{\mathbb{R}^N}&(I_\alpha\ast|u|^p)|u|^{p-2}u2\theta_{\epsilon}\nabla\bar{u}\cdot\nabla
wdx\\
&=\int_{\mathbb{R}^N}\theta_{\epsilon}(I_\alpha\ast|u|^p)|u|^{p-2}(u\nabla\bar{u}+\bar{u}\nabla
u)\cdot\nabla
wdx\\
&=\frac{2}{p}\int_{\mathbb{R}^N}\theta_{\epsilon}(I_\alpha\ast|u|^p)\nabla(|u|^p)\cdot\nabla
wdx\\
&=-\frac{2}{p}\int_{\mathbb{R}^N}\left[(I_\alpha\ast|u|^p)|u|^p\nabla\theta_{\epsilon}\cdot\nabla
w+(I_\alpha\ast|u|^p)|u|^p\theta_{\epsilon}\Delta
w\right.\\
&\qquad\qquad\qquad\left.+\theta_{\epsilon}|u|^p\nabla(I_\alpha\ast|u|^p)\cdot\nabla
w\right]dx,
\end{split}
\end{equation}
where
\begin{equation}\label{e5.9}
\begin{split}
&\int_{\mathbb{R}^N}\theta_{\epsilon}|u|^p\nabla(I_\alpha\ast|u|^p)\cdot\nabla
wdx\\
&=\mathcal{A}\int_{\mathbb{R}^N}\nabla\left(\int_{\mathbb{R}^N}\frac{|u(y)|^p}{|x-y|^{N-\alpha}}dy\right)\cdot
\nabla w(x)|u(x)|^p\theta_{\epsilon}(x)dx\\
&=\mathcal{A}\int_{\mathbb{R}^N}\int_{\mathbb{R}^N}
\frac{(\alpha-N)(x-y)\cdot\nabla
w(x)\theta_\epsilon(x)|u(x)|^p|u(y)|^p}{|x-y|^{N-\alpha+2}}dydx\\
&=\frac{(\alpha-N)\mathcal{A}}{2}\left(\int_{\mathbb{R}^N}\int_{\mathbb{R}^N}
\frac{(x-y)\cdot\nabla
w(x)\theta_\epsilon(x)|u(x)|^p|u(y)|^p}{|x-y|^{N-\alpha+2}}dydx\right.\\
&\qquad\qquad\qquad\qquad+\left.\int_{\mathbb{R}^N}\int_{\mathbb{R}^N}
\frac{(y-x)\cdot\nabla
w(y)\theta_\epsilon(y)|u(x)|^p|u(y)|^p}{|x-y|^{N-\alpha+2}}dydx\right)\\
&=\frac{(\alpha-N)\mathcal{A}}{2}\int_{\mathbb{R}^N}\int_{\mathbb{R}^N}
\frac{(x-y)\cdot(\theta_\epsilon(x)\nabla
w(x)-\theta_\epsilon(y)\nabla
w(y))|u(x)|^p|u(y)|^p}{|x-y|^{N-\alpha+2}}dydx.
\end{split}
\end{equation}
Inserting (\ref{e5.9}) into (\ref{e5.8}), we obtain that
\begin{equation}\label{e5.10}
\begin{split}
&\mathrm{Re}\int_{\mathbb{R}^N}(I_\alpha\ast|u|^p)|u|^{p-2}u2\theta_{\epsilon}\nabla\bar{u}\cdot\nabla
wdx\\
&=-\frac{2}{p}\int_{\mathbb{R}^N}(I_\alpha\ast|u|^p)|u|^p\left(\nabla\theta_{\epsilon}\cdot\nabla
w+\theta_{\epsilon}\Delta
w\right)dx\\
&+\frac{(N-\alpha)\mathcal{A}}{p}\int_{\mathbb{R}^N}\int_{\mathbb{R}^N}
\frac{(x-y)\cdot(\theta_\epsilon(x)\nabla
w(x)-\theta_\epsilon(y)\nabla
w(y))|u(x)|^p|u(y)|^p}{|x-y|^{N-\alpha+2}}dydx.
\end{split}
\end{equation}
Inserting (\ref{e5.4})-(\ref{e5.7}) and (\ref{e5.10}) into
(\ref{e5.3}), we obtain that
\begin{equation}\label{e5.11}
\begin{split}
h_\epsilon'(t)&=\mathrm{Re}\int_{\mathbb{R}^N}\left(2\nabla
u\cdot\nabla\theta_{\epsilon}\nabla\bar{u}\cdot\nabla
w+2\theta_{\epsilon}\partial_ju\partial_i\bar{u}\partial_{ij}w+\nabla
u\cdot\nabla\theta_{\epsilon}\bar{u}\Delta
w \right.\\
&
\qquad\qquad\qquad-\frac{1}{2}|u|^2\nabla\theta_{\epsilon}\cdot\nabla(\Delta
w)-\frac{1}{2}|u|^2\theta_{\epsilon}\Delta^2w
+\bar{u}\partial_ju\partial_{ij}\theta_{\epsilon}\partial_iw\\
&\left.\qquad\qquad\qquad+\bar{u}\partial_ju\partial_{i}\theta_{\epsilon}\partial_{ij}w+\frac{2b\theta_{\epsilon}x\cdot\nabla
w}{(|x|^2+\delta)^2}|u|^2 \right)dx\\
&\quad
+\left(\frac{2}{p}-1\right)\int_{\mathbb{R}^N}(I_\alpha\ast|u|^p)|u|^p\left(\nabla\theta_{\epsilon}\cdot\nabla
w+\theta_{\epsilon}\Delta
w\right)dx\\
&\quad
-\frac{(N-\alpha)\mathcal{A}}{p}\int_{\mathbb{R}^N}\int_{\mathbb{R}^N}
\frac{(x-y)\cdot(\theta_\epsilon(x)\nabla
w(x)-\theta_\epsilon(y)\nabla
w(y))|u(x)|^p|u(y)|^p}{|x-y|^{N-\alpha+2}}dydx.
\end{split}
\end{equation}
Note that $\theta_{\epsilon}$, $\partial_i\theta_{\epsilon}$,
$\partial_{ij}\theta_{\epsilon}$ are bounded with respect to both
$x$ and $\epsilon$, and by the mean value theorem,
\begin{equation*}
(x-y)\cdot(\theta_\epsilon(x)\nabla w(x)-\theta_\epsilon(y)\nabla
w(y))\lesssim |x-y|^2.
\end{equation*}
Furthermore, $\theta_\epsilon\to 1$, $\partial_i\theta_{\epsilon}\to
0$, $\partial_{ij}\theta_{\epsilon}\to 0$ and
\begin{equation*}
(x-y)\cdot(\theta_\epsilon(x)\nabla w(x)-\theta_\epsilon(y)\nabla
w(y))\to (x-y)\cdot(\nabla w(x)-\nabla w(y))
\end{equation*}
as $\epsilon\to 0$. On the other hand, for any $t\in I$, we have
$u(t)\in H^2(\mathbb{R}^N)$ and $\sqrt{w(x)}u(t)\in
L^2(\mathbb{R}^N)$, so by the choices of $w(x)$ and
$\theta_\epsilon(x)$ and by the Lebesgue dominated convergence
theorem, we obtain from (\ref{e5.11}) that
\begin{equation*}
\begin{split}
\lim_{\epsilon\to
0}h_\epsilon'(t)&=2\mathrm{Re}\int_{\mathbb{R}^N}\partial_{i}u\partial_j\bar{u}\partial_{ij}
udx-\frac{1}{2}\int_{\mathbb{R}^N}|u|^2\Delta^2wdx+2b\int_{\mathbb{R}^N}\frac{x\cdot\nabla
w}{(|x|^2+\delta)^2}|u|^2dx\\
&\
-\frac{\mathcal{A}(N-\alpha)}{p}\int_{\mathbb{R}^N}\int_{\mathbb{R}^N}\frac{(x-y)\cdot(\nabla
w(x)-\nabla w(y))|u(x)|^p|u(y)|^p}{|x-y|^{N-\alpha+2}}dxdy\\
&\
+\left(\frac{2}{p}-1\right)\int_{\mathbb{R}^N}(I_\alpha\ast|u|^p)|u|^p\Delta
wdx.
\end{split}
\end{equation*}
Since
\begin{equation*}
\lim_{\epsilon\to
0}2h_{\epsilon}(t)=2\mathrm{Im}\int_{\mathbb{R}^N}\bar{u}\nabla
w\cdot\nabla udx=\frac{d}{dt}V_{w}(t),
\end{equation*}
we see that $V_{w}(t)\in C^2(I,\mathbb{R})$  and  (\ref{e5.2})
holds.

Step 2. Let $\{u_{n0}\}_{n=1}^{\infty}\subset H^2(\mathbb{R}^N)$ be
such that $u_{n0}\to u_0$ strongly in $H^1(\mathbb{R}^N)$ and
$\sqrt{w(x)}u_{n0}\to \sqrt{w(x)}u_{0}$ strongly in
$L^2(\mathbb{R}^N)$ as $n\to\infty$, and let $u_{n}$ be the
corresponding solution to $\mathrm{(CH_b^\delta)}$ with initial
value $u_{n0}$. Let $\Phi(t)$ denote the right-hand side of
(\ref{e5.2}) and let $\Phi_n(t)$ denote the right-hand side of
(\ref{e5.2}) corresponding to the solution $u_{n}$. It follows from
Step 1 that
\begin{equation}\label{e5.12}
\|\sqrt{w(x)}u_n(t)\|_{L^2}^2=\|\sqrt{w(x)}u_{n0}\|_{L^2}^2+2t\mathrm{Im}\int_{\mathbb{R}^N}\bar{u}_{n0}\nabla
w\cdot\nabla u_{n0}dx+\int_{0}^t\int_{0}^{s}\Phi_n(\tau)d\tau ds.
\end{equation}
By the continuous dependence (Proposition \ref{pro5.1}) and
Corollary 6.5.3 in \cite{Cazenave 2003}, we may let $n\to \infty$ in
(\ref{e5.12}) and obtain
\begin{equation}\label{e5.12'}
\|\sqrt{w(x)}u(t)\|_{L^2}^2=\|\sqrt{w(x)}u_{0}\|_{L^2}^2+2t\mathrm{Im}\int_{\mathbb{R}^N}\bar{u}_{0}\nabla
w\cdot\nabla u_{0}dx+\int_{0}^t\int_{0}^{s}\Phi(\tau)d\tau ds,
\end{equation}
which implies (\ref{e5.2}).
\end{proof}

\begin{lemma}\label{lem6.1}
Let $N\geq 3$, $\alpha\in ((N-4)_+,N)$, $b>-\frac{(N-2)^2}{4}$ and
$b\neq0$, $2\leq p<\frac{N+\alpha}{N-2}$. If $u\in
C(I,H^1(\mathbb{R}^N))\cap C^1(I,H^{-1}(\mathbb{R}^N))$ is a
maximal-lifespan solution to $\mathrm{(CH_b)}$ with initial value
$u_0\in H^1(\mathbb{R}^N)$ and $\sqrt{w(x)}u_0\in
L^2(\mathbb{R}^N)$, then for any $t\in I$,
\begin{equation*}
\frac{d}{dt}V_w(t)=2\mathrm{Im}\int_{\mathbb{R}^N}\bar{u}\nabla
w\cdot\nabla udx
\end{equation*}
and
\begin{equation*}
\begin{split}
\frac{d^2}{dt^2}V_w(t)&=4\mathrm{Re}\int_{\mathbb{R}^N}\partial_i
u\partial_j \bar{u}\partial_{ij}
wdx-\int_{\mathbb{R}^N}|u|^2\Delta^2wdx+4b\int_{\mathbb{R}^N}|x|^{-4}|u|^2x\cdot\nabla
wdx\\
&\
-\frac{2\mathcal{A}(N-\alpha)}{p}\int_{\mathbb{R}^N}\int_{\mathbb{R}^N}\frac{(x-y)\cdot(\nabla
w(x)-\nabla w(y))|u(x)|^p|u(y)|^p}{|x-y|^{N-\alpha+2}}dxdy\\
&\
+\left(\frac{4}{p}-2\right)\int_{\mathbb{R}^N}(I_\alpha\ast|u|^p)|u|^p\Delta
wdx.
\end{split}
\end{equation*}
\end{lemma}

\begin{proof}
Inspired by the proof of (3.1) in \cite{Suzuki 2014}, we prove this
lemma in two steps. For any $\delta>0$, let $u_\delta\in
C(I_{\delta}, H^1(\mathbb{R}^N))\cap C^1(I_{\delta},
H^{-1}(\mathbb{R}^N))$ be the maximal-lifespan solution to
$\mathrm{(CH_b^\delta)}$ with initial value $u_0$.

Step 1.  We claim  that there exists $M_0>0$ and $T>0$ such that
\begin{equation}\label{e5.17}
\|u_\delta(t)\|_{H^1}\leq M_0 \ \mathrm{for\ any\ }\delta>0 \
\mathrm{and}\ t\in [-T,T].
\end{equation}
Indeed, we define
\begin{equation*}
M:=2\sqrt{\left(1+\frac{4b_+}{(N-2)^2}\right)}\,\Large{/}\sqrt{\left(1-\frac{4b_-}{(N-2)^2}\right)}\,\|u_0\|_{H^1}
\end{equation*}
and
\begin{equation*}
\tau_{\delta}:=\sup_{T>0}\{\|u_\delta(t)\|_{H^1}\leq M,\ t\in
[-T,T]\}.
\end{equation*}
If $\tau_{\delta}=\infty$, the claim is proved. Thus, we assume
$\tau_{\delta}<\infty$. Since $u_\delta\in C(I_\delta,
H^1(\mathbb{R}^N))$, $\tau_\delta$ satisfies
\begin{equation}\label{e5.34}
\|u_\delta(\tau_\delta)\|_{H^1}=M\ \mathrm{or}\
\|u_\delta(-\tau_\delta)\|_{H^1}=M.
\end{equation}
It follows from the conservation laws, (G3) and (\ref{e5.13}) that
\begin{equation}\label{e5.14}
\begin{split}
&\left(1-\frac{4b_-}{(N-2)^2}\right)\|u_\delta\|_{H^1}^2
-\left(1+\frac{4b_+}{(N-2)^2}\right)\|u_0\|_{H^1}^2\\
&\leq \int_{\mathbb{R}^N}(|\nabla
u_\delta|^2+\frac{b}{|x|^2+\delta}|u_\delta|^2+|u_\delta|^2)dx\\
&\qquad-\int_{\mathbb{R}^N}(|\nabla
u_0|^2+\frac{b}{|x|^2+\delta}|u_0|^2+|u_0|^2)dx\\
&=\frac{1}{p}\int_{\mathbb{R}^N}(I_\alpha\ast|u_\delta|^p)|u_\delta|^pdx
-\frac{1}{p}\int_{\mathbb{R}^N}(I_\alpha\ast|u_0|^p)|u_0|^pdx\\
&\lesssim \epsilon+C_\epsilon(M)\|u_\delta-u_0\|_{L^2}
\end{split}
\end{equation}
for any $t\in[-\tau_\delta,\tau_\delta]$ and $\epsilon>0$. On the
other hand, in view of (G2) and the fact that $u_\delta$ satisfies
equation $\mathrm{(CH_b^\delta)}$, we know
\begin{equation}\label{e5.15}
\begin{split}
\|\partial_tu_\delta\|_{H^{-1}}&\leq \|-\Delta
u_\delta+\frac{b}{|x|^2+\delta}u_\delta\|_{H^{-1}}+\|(I_\alpha\ast|u_\delta|^p)|u_\delta|^{p-2}u_\delta\|_{H^{-1}}\\
&\leq C(M)
\end{split}
\end{equation}
for any $t\in[-\tau_\delta,\tau_\delta]$. Applying Lemma 3.3.6 in
\cite{Cazenave 2003}, we obtain that
\begin{equation}\label{e5.16}
\|u_\delta(t)-u_\delta(s)\|_{L^2}\leq C(M)|t-s|^{\frac{1}{2}}\
\mathrm{for\ any\ }t,\ s\in [-\tau_\delta,\tau_\delta].
\end{equation}
Inserting (\ref{e5.16}) with $s=0$ into (\ref{e5.14}), we have
\begin{equation*}
\left(1-\frac{4b_-}{(N-2)^2}\right)\|u_\delta\|_{H^1}^2
-\left(1+\frac{4b_+}{(N-2)^2}\right)\|u_0\|_{H^1}^2\leq
\epsilon+C_\epsilon(M)C(M)|t|^{\frac{1}{2}}.
\end{equation*}
Letting $t=\tau_{\delta}$ or $-\tau_{\delta}$ and applying
(\ref{e5.34}), we have
\begin{equation*}
3\left(1+\frac{4b_+}{(N-2)^2}\right)\|u_0\|_{H^1}^2\leq
\epsilon+C_\epsilon(M)C(M)|\tau_\delta|^{\frac{1}{2}},
\end{equation*}
which implies that $\tau_\delta$ has a positive lower bound by
choosing $\epsilon>0$ small enough. The  proof of the claim is
complete.

Step 2. In view of (\ref{e5.17}) and (\ref{e5.15}), by Proposition
1.1.2 in \cite{Cazenave 2003}, there exists
$\{\delta_j\}_{j=1}^{\infty}\subset (0,\infty)$ with $\delta_j\to 0$
as $j\to \infty$ and $v\in C_{w}([-T,T], H^1(\mathbb{R}^N))$ $\cap
W^{1,\infty}([-T,T], H^{-1}(\mathbb{R}^N))$ such that
\begin{equation}\label{e5.22}
u_{\delta_j}(t)\to v(t)\ \mathrm{weakly\ in\ }H^1(\mathbb{R}^N) \
\mathrm{for\ any\ }t\in[-T,T]
\end{equation}
and
\begin{equation}\label{e5.23}
\partial_tu_{\delta_j}(t)\to \partial_tv(t)\ \mathrm{weakly *\ in\
}L^{\infty}([-T,T], H^{-1}(\mathbb{R}^N)).
\end{equation}
By (\ref{e5.21}) and  (\ref{e5.22}), we have
\begin{equation}\label{e5.24}
\left(-\Delta +\frac{b}{|x|^2+\delta_j}\right)u_{\delta_j}(t)\to
\left(-\Delta +\frac{b}{|x|^2}\right)v(t)\ \mathrm{weakly\ in\
}H^{-1}(\mathbb{R}^N)
\end{equation}
as $j\to\infty$ for  any $t\in[-T,T]$. By (\ref{e5.17}),
(\ref{e5.23}), (\ref{e5.24}) and the fact that $u_{\delta_j}$
satisfies equation $\mathrm{(CH_b^{\delta_j})}$, there exists $f$
such that
\begin{equation}\label{e5.37}
\begin{split}
-(I_\alpha\ast|u_{\delta_j}|^p)|u_{\delta_j}|^{p-2}u_{\delta_j}
&=i\partial_tu_{\delta_j}-(-\Delta+\frac{b}{|x|^2+\delta_j})u_{\delta_j}\\
&\to i\partial_tv-\mathcal{L}_bv=:f\ \mathrm{weakly*\ in }\
L^{\infty}([-T,T],H^{-1}(\mathbb{R}^N)).
\end{split}
\end{equation}
By (G5), (\ref{e5.22}) and (\ref{e5.37}), we have
\begin{equation}\label{e5.36}
\begin{split}
&\mathrm{Im}\int_{0}^{t}\langle f(s),v(s)\rangle_{H^{-1},H^1}ds\\
&=\lim_{n\to \infty}\mathrm{Im}\int_{0}^{t}\langle
-(I_\alpha\ast|u_{\delta_j}(s)|^p)|u_{\delta_j}(s)|^{p-2}u_{\delta_j}(s),u_{\delta_j}(s)\rangle_{H^{-1},H^1}ds\\
&=0,\ \mathrm{for\ any\ } t\in[-T,T].
\end{split}
\end{equation}
By  (\ref{e5.37}) and (\ref{e5.36}), we obtain the conservation of
mass  of $v$. Hence, it follows from (\ref{e5.22}), the conservation
laws of $u_{\delta_j}$ and $u_{\delta_j}(0)=u_0$ that
\begin{equation}\label{e5.25}
\|v(t)\|_{L^2}=\|v(0)\|_{L^2}=\|u_0\|_{L^2}=\|u_{\delta_j}\|_{L^2},\
\ \mathrm{for\ any\ }  t\in[-T,T].
\end{equation}
Hence we see from (\ref{e5.22}) and (\ref{e5.25}) that
\begin{equation*}
u_{\delta_j}(t)\to v(t)\ \mathrm{strongly\ in\ }L^2(\mathbb{R}^N)\
\mathrm{for\ any\ }t\in[-T,T]\ \mathrm{as}\ j\to\infty,
\end{equation*}
which combines with (\ref{e5.17}) and the Sobolev imbedding theorem
shows that
\begin{equation}\label{e5.*4}
u_{\delta_j}(t)\to v(t)\ \mathrm{strongly\ in\ }L^r(\mathbb{R}^N)\
 \mathrm{for\ any\ }r\in[2,2^*)\ \mathrm{and}\ t\in[-T,T].
\end{equation}
By (\ref{e5.17}), (\ref{e5.37}) and (\ref{e5.*4}), we have
$f=-(I_\alpha\ast|v|^p)|v|^{p-2}v$. Thus, $v$ satisfies
\begin{equation}\label{e5*3}
\begin{cases}
i\partial_tv-\mathcal{L}_bv=-(I_{\alpha}\ast|v|^{p})|v|^{p-2}v,\
\mathrm{in}\ L^{\infty}([-T,T],H^{-1}(\mathbb{R}^{N})),\\
v(0,x)=u_0(x),\quad x\in \mathbb{R}^{N}.
\end{cases}
\end{equation}
On the other hand, there exists a unique weak solution to
(\ref{e5*3}), see the proof of Theorem \ref{thm local}. Hence, $v=u$
and $u$ satisfies the conservation laws
\begin{equation}\label{e5.*5}
\|u(t)\|_{L^2}=\|u_0\|_{L^2},\ E_b(u(t))=E_b(u_0),\ \mathrm{for\
any}\ t\in[-T,T].
\end{equation}
By (\ref{e5.*4}), (\ref{e5.*5}), the conservation laws of
$u_{\delta_j}$, we have
\begin{equation*}
\begin{split}
\|\nabla
u_{\delta_j}\|_{L^2}^2+b\left\|\frac{u_{\delta_j}}{\sqrt{|x|^2+\delta_j}}\right\|_{L^2}^2
&=2E_{b,\delta_j}(u_{\delta_j}(t))+\frac{1}{p}\int_{\mathbb{R}^N}(I_\alpha\ast|u_{\delta_j}|^p)|u_{\delta_j}|^pdx\\
&\to2E_{b,\delta_j}(u_0)+\frac{1}{p}\int_{\mathbb{R}^N}(I_\alpha\ast|u|^p)|u|^pdx\\
&\to 2E_{b}(u_0)+\frac{1}{p}\int_{\mathbb{R}^N}(I_\alpha\ast|u|^p)|u|^pdx\\
&=\|\nabla
u\|_{L^2}^2+b\left\|\frac{u_{\delta_j}}{|x|}\right\|_{L^2}^2.
\end{split}
\end{equation*}
By Lemma 2.2 in \cite{Suzuki 2014}, we obtain $\nabla
u_{\delta_j}\to \nabla u$ strongly in $L^2(\mathbb{R}^N)$. Hence,
\begin{equation}\label{e5.28}
u_{\delta_j}\to u\ \mathrm{strongly\ in\ } H^1(\mathbb{R}^N)\
\mathrm{for\ any\ }t\in[-T,T].
\end{equation}
On the other hand, Lemma 6.5.2 in \cite{Cazenave 2003} implies
\begin{equation*}
\int_{\mathbb{R}^N}w|u|^2dx=\int_{\mathbb{R}^N}w|u_0|^2dx+2\mathrm{Im}\int_{0}^{t}\int_{\mathbb{R}^N}\bar{u}\nabla
w\cdot\nabla udx.
\end{equation*}
Thus Lemma 2.3 in \cite{Suzuki 2014} with (\ref{e5.17}) and
(\ref{e5.28}) give that
\begin{equation*}
\sqrt{w(x)}u_{\delta_j}(t)\to \sqrt{w(x)}u(t) \ \mathrm{strongly\
in\ } L^2(\mathbb{R}^N)\ \mathrm{for\ any\ }t\in[-T,T].
\end{equation*}
Replacing $u$ by $u_\delta$ in (\ref{e5.12'}) and letting $\delta\to
0$, we complete the proof.
\end{proof}

Let $w=|x|^2$ in Lemma \ref{lem6.1}, then
\begin{equation*}
\nabla w=2x,\ \Delta w=2N,\ \partial_{ij}w=2\delta_{ij},\
\Delta^2w=0,\ \partial_iu\partial_j\bar{u}\partial_{ji}w=2|\nabla
u|^2,
\end{equation*}
\begin{equation*}
x\cdot\nabla w=2|x|^2,\ (x-y)\cdot(\nabla w(x)-\nabla
w(y))=2|x-y|^2.
\end{equation*}
Hence, we have the following standard virial identity.

\begin{lemma}\label{lem6.2}
Let $N\geq 3$, $\alpha\in((N-4)_+,N)$, $b>-\frac{(N-2)^2}{4}$ and
$b\neq0$, $2\leq p<\frac{N+\alpha}{N-2}$, $u_0\in H^1(\mathbb{R}^N)$
with $xu_0\in L^2(\mathbb{R}^N)$. Assume that $u(x,t)\in
C(I,H^1(\mathbb{R}^N))\cap C^1(I,H^{-1}(\mathbb{R}^N))$ is the
maximal-lifespan  solution to $\mathrm{(CH_b)}$. Then $xu\in
C(I,L^2(\mathbb{R}^N))$ and  for any $t\in I$,
\begin{equation*}
\frac{d^2}{dt^2}\|xu(t)\|_{L^2}^2=8\|u\|_{\dot{H}_b^1}^2+\frac{4\alpha+4N-4Np}{p}\int_{\mathbb{R}^N}(I_\alpha\ast|u|^p)|u|^pdx.
\end{equation*}
\end{lemma}

Let $\psi:[0,\infty)\to [0,\infty)$ be a smooth function satisfying
\begin{equation*}
\psi(r)=\left\{\begin{array}{ll} r^2,&\ 0\leq r\leq 1,\\
\mathrm{constant},&\ r\geq 10,
\end{array}\right.\ \psi'(r)\leq 2r,\ \psi''(r)\leq 2\ \mathrm{for\ any}\
r\geq 0.
\end{equation*}
For any $R>1$, we define
\begin{equation}\label{e5.61}
\psi_R(r)=R^2\psi(\frac{r}{R}),\  \varphi_R(x)=\psi_R(|x|).
\end{equation}

By letting  $w(x)=\varphi_R(x)$ in Lemma \ref{lem6.1}, we have the
following local virial identity.

\begin{lemma}\label{lem6.3}
Let $N\geq 3$, $\alpha\in ((N-4)_+,N)$, $b>-\frac{(N-2)^2}{4}$ and
$b\neq0$, $2\leq p<p^b$, $u_0\in H_r^1(\mathbb{R}^N)$ and $u\in
C(I,H^1(\mathbb{R}^N))\cap C^1(I,H^{-1}(\mathbb{R}^N))$ be the
maximal-lifespan radial solution to $\mathrm{(CH_b)}$. Then for any
$t\in I$,
\begin{equation*}
\begin{split}
\frac{d^2}{dt^2}\|\varphi_R(x)u\|_{L^2}^2&=8\|u\|_{\dot{H}_b^1}^2+\frac{4\alpha+4N-4Np}{p}\int_{\mathbb{R}^N}(I_\alpha\ast|u|^p)|u|^pdx\\
&\quad +
O(R^{-2})+O(R^{-(N-1)(p-\frac{N+\alpha}{N})}\|u\|_{\dot{H}_b^1}^{p-\frac{N+\alpha}{N}})\\
&\quad+
O(R^{-(N-\alpha+\frac{N-1}{2}(p-2))}\|u\|_{\dot{H}_b^1}^{\frac{N+1}{2}(p-2)}).
\end{split}
\end{equation*}
\end{lemma}

\begin{proof}
By Lemma \ref{lem6.1}, we have
\begin{equation}\label{e5.29'}
\begin{split}
&\frac{d^2}{dt^2}\|\varphi_R(x)u\|_{L^2}^2\\
&=4\mathrm{Re}\int_{\mathbb{R}^N}\partial_i u\partial_j
\bar{u}\partial_{ij}
\varphi_Rdx-\int_{\mathbb{R}^N}|u|^2\Delta^2\varphi_Rdx+4b\int_{\mathbb{R}^N}|x|^{-4}|u|^2x\cdot\nabla
\varphi_Rdx\\
&\qquad
-\frac{2\mathcal{A}(N-\alpha)}{p}\int_{\mathbb{R}^N}\int_{\mathbb{R}^N}\frac{(x-y)\cdot(\nabla
\varphi_R(x)-\nabla \varphi_R(y))|u(x)|^p|u(y)|^p}{|x-y|^{N-\alpha+2}}dxdy\\
&\qquad+\left(\frac{4}{p}-2\right)\int_{\mathbb{R}^N}(I_\alpha\ast|u|^p)|u|^p\Delta
\varphi_Rdx.
\end{split}
\end{equation}
Direct calculation gives that
\begin{equation}\label{e6.11}
\begin{split}
4\mathrm{Re}\int_{\mathbb{R}^N}\partial_i u\partial_j
\bar{u}\partial_{ij}
\varphi_Rdx&=4\mathrm{Re}\int_{\mathbb{R}^N}\partial_i u\partial_j
\bar{u}\partial_{ij}
(|x|^2)dx\\
&\qquad+4\mathrm{Re}\int_{|x|>R}\partial_i u\partial_j
\bar{u}\left(\partial_{ij}\varphi_R-\partial_{ij}(|x|^2)\right)dx,
\end{split}
\end{equation}
\begin{equation}\label{e6.12}
\begin{split}
\int_{\mathbb{R}^N}|u|^2\Delta^2\varphi_Rdx=\int_{\mathbb{R}^N}|u|^2\Delta^2(|x|^2)dx
+\int_{|x|>R}|u|^2(\Delta^2\varphi_R-\Delta^2(|x|^2))dx,
\end{split}
\end{equation}
\begin{equation}\label{e6.13}
\begin{split}
4b\int_{\mathbb{R}^N}|x|^{-4}|u|^2x\cdot\nabla
\varphi_Rdx&=4b\int_{\mathbb{R}^N}|x|^{-4}|u|^2x\cdot\nabla
(|x|^2)dx\\
&\ +4b\int_{|x|>R}|x|^{-4}|u|^2(x\cdot\nabla \varphi_R-x\cdot\nabla
(|x|^2))dx,
\end{split}
\end{equation}
\begin{equation}\label{e6.14}
\begin{split}
\left(\frac{4}{p}-2\right)&\int_{\mathbb{R}^N}(I_\alpha\ast|u|^p)|u|^p\Delta
\varphi_Rdx\\
&=\left(\frac{4}{p}-2\right)\int_{\mathbb{R}^N}(I_\alpha\ast|u|^p)|u|^p\Delta
(|x|^2)dx\\
&\
+\left(\frac{4}{p}-2\right)\int_{|x|>R}(I_\alpha\ast|u|^p)|u|^p(\Delta
\varphi_R-\Delta (|x|^2))dx
\end{split}
\end{equation}
and
\begin{equation}\label{e6.15}
\begin{split}
&\int_{\mathbb{R}^N}\int_{\mathbb{R}^N}\frac{(x-y)\cdot(\nabla
\varphi_R(x)-\nabla
\varphi_R(y))|u(x)|^p|u(y)|^p}{|x-y|^{N-\alpha+2}}dxdy\\
&=\int_{\mathbb{R}^N}\int_{\mathbb{R}^N}\frac{(x-y)\cdot(\nabla
|x|^2-\nabla
|y|^2)|u(x)|^p|u(y)|^p}{|x-y|^{N-\alpha+2}}dxdy\\
&-\int_{|x|>R}\int_{\mathbb{R}^N}\frac{(x-y)\cdot(\nabla
|x|^2-\nabla |y|^2)|u(x)|^p|u(y)|^p}{|x-y|^{N-\alpha+2}}dxdy\\
&-\int_{|x|<R}\int_{|y|>R}\frac{(x-y)\cdot(\nabla
|x|^2-\nabla |y|^2)|u(x)|^p|u(y)|^p}{|x-y|^{N-\alpha+2}}dxdy\\
&+\int_{|x|<R}\int_{|y|>R}\frac{(x-y)\cdot(\nabla
\varphi_R(x)-\nabla
\varphi_R(y))|u(x)|^p|u(y)|^p}{|x-y|^{N-\alpha+2}}dxdy\\
&+\int_{|x|>R}\int_{\mathbb{R}^N}\frac{(x-y)\cdot(\nabla
\varphi_R(x)-\nabla
\varphi_R(y))|u(x)|^p|u(y)|^p}{|x-y|^{N-\alpha+2}}dxdy.
\end{split}
\end{equation}
Inserting (\ref{e6.11})-(\ref{e6.15}) into (\ref{e5.29'}), we obtain
that
\begin{equation}\label{e5.341}
\begin{split}
&\frac{d^2}{dt^2}\|\varphi_R(x)u\|_{L^2}^2\\
&=8\|u\|_{\dot{H}_b^1(\mathbb{R}^N)}^2+\frac{4\alpha+4N-4Np}{p}\int_{\mathbb{R}^N}(I_\alpha\ast|u|^p)|u|^pdx\\
&-8\|u\|_{\dot{H}_b^1(B_R^c)}^2-\frac{4\alpha+4N-4Np}{p}\int_{|x|>R}(I_\alpha\ast|u|^p)|u|^pdx\\
&+4\mathrm{Re}\int_{|x|>R}\partial_iu\partial_j\bar{u}\partial_{ij}\varphi_Rdx-\int_{|x|>R}|u|^2\Delta^2\varphi_Rdx\\
&+4b\int_{|x|>R}|x|^{-4}|u|^2x\cdot\nabla
\varphi_Rdx+\left(\frac{4}{p}-2\right)\int_{|x|>R}(I_\alpha\ast|u|^p)|u|^p\Delta
\varphi_Rdx\\
&+\frac{2\mathcal{A}(\alpha-N)}{p}\int_{|x|>R}\int_{\mathbb{R}^N}\frac{(x-y)\cdot(\nabla
\varphi_R(x)-\nabla
\varphi_R(y))|u(x)|^p|u(y)|^p}{|x-y|^{N-\alpha+2}}dxdy\\
&-\frac{2\mathcal{A}(\alpha-N)}{p}\int_{|x|<R}\int_{|y|>R}\frac{(x-y)\cdot(\nabla
|x|^2-\nabla |y|^2)|u(x)|^p|u(y)|^p}{|x-y|^{N-\alpha+2}}dxdy\\
&+\frac{2\mathcal{A}(\alpha-N)}{p}\int_{|x|<R}\int_{|y|>R}\frac{(x-y)\cdot(\nabla
\varphi_R(x)-\nabla
\varphi_R(y))|u(x)|^p|u(y)|^p}{|x-y|^{N-\alpha+2}}dxdy.
\end{split}
\end{equation}
By direct calculation, we have
\begin{equation*}
\psi_R'(r)=R\psi'(\frac{r}{R}),\ \psi_R''(r)=\psi''(\frac{r}{R}),
\end{equation*}
\begin{equation*}
2-\psi_R''(r)\geq 0,\
2-\frac{\psi_R'(r)}{r}=2-\frac{R}{r}\psi'(\frac{r}{R})\geq 0,
\end{equation*}
\begin{equation*}
2N-\Delta\varphi_R(x)=2N-\left(\psi_R''(|x|)+\frac{N-1}{|x|}\psi_R'(|x|)\right)\geq
0,
\end{equation*}
\begin{equation*}
\begin{split}
&\partial_i\varphi_R(x)=R\psi'(\frac{|x|}{R})\frac{x_i}{|x|},\
x\cdot\nabla \varphi_R(x)=R|x|\psi'(\frac{|x|}{R}),\\
&\partial_{ij}\varphi_R(x)=\psi''(\frac{|x|}{R})\frac{x_ix_j}{|x|^2}
+R\psi'(\frac{|x|}{R})\frac{\delta_{ij}}{|x|}-R\psi'(\frac{|x|}{R})\frac{x_ix_j}{|x|^3}\lesssim 1,\\
&\Delta
\varphi_R(x)=\psi''(\frac{|x|}{R})+(N-1)\frac{R}{|x|}\psi'(\frac{|x|}{R})\lesssim1,\
\Delta^2 \varphi_R(x)\lesssim R^{-2},
\end{split}
\end{equation*}
\begin{equation*}
\begin{split}
\partial_iu\partial_j\bar{u}\partial_{ij}\varphi_R(x)
&=|\partial_ru|^2\left(\psi''(\frac{|x|}{R})+\frac{R}{|x|}\psi'(\frac{|x|}{R})-\frac{R}{|x|}\psi'(\frac{|x|}{R})\right)\\
&=|\nabla u|^2 \psi''(\frac{|x|}{R}),
\end{split}
\end{equation*}
and
\begin{equation*}
\begin{split}
(x-y)\cdot(\nabla \varphi_R(x)-\nabla \varphi_R(y))\lesssim |x-y|^2.
\end{split}
\end{equation*}
The above estimates and the conservation laws imply that
\begin{equation}\label{e5.344}
\int_{|x|>R}|u|^2\Delta^2\varphi_Rdx\lesssim
\int_{|x|>R}|u|^2R^{-2}dx =O(R^{-2}),
\end{equation}
\begin{equation}\label{e5.342}
\begin{split}
&4\mathrm{Re}\int_{|x|>R}\partial_iu\partial_j\bar{u}\partial_{ij}\varphi_Rdx+4b\int_{|x|>R}|x|^{-4}|u|^2x\cdot\nabla
\varphi_Rdx-8\|u\|_{\dot{H}_b^1(B_R^c)}^2\\
&=4\int_{|x|>R}\psi''(\frac{|x|}{R})|\nabla
u|^2dx+4b\int_{|x|>R}|x|^{-2}|u|^2\frac{R}{|x|}\psi'(\frac{|x|}{R})dx-8\|u\|_{\dot{H}_b^1(B_R^c)}^2\\
&=4\int_{|x|>R}\left(\psi''(\frac{|x|}{R})-2\right)|\nabla
u|^2dx+4b\int_{|x|>R}|x|^{-2}|u|^2\left(\frac{R}{|x|}\psi'(\frac{|x|}{R})-2\right)dx\\
&\lesssim R^{-2}
\end{split}
\end{equation}
and
\begin{equation}\label{e5.343}
\begin{split}
&\int_{|x|>R}(I_\alpha\ast|u|^p)|u|^p \Delta \varphi_Rdx\\
&+\int_{|x|>R}\int_{\mathbb{R}^N}\frac{(x-y)\cdot(\nabla
\varphi_R(x)-\nabla
\varphi_R(y))|u(x)|^p|u(y)|^p}{|x-y|^{N-\alpha+2}}dxdy\\
&+\int_{|x|<R}\int_{|y|>R}\frac{(x-y)\cdot(\nabla
|x|^2-\nabla |y|^2)|u(x)|^p|u(y)|^p}{|x-y|^{N-\alpha+2}}dxdy\\
&+\int_{|x|<R}\int_{|y|>R}\frac{(x-y)\cdot(\nabla
\varphi_R(x)-\nabla
\varphi_R(y))|u(x)|^p|u(y)|^p}{|x-y|^{N-\alpha+2}}dxdy\\
&\lesssim\int_{|x|>R}(I_\alpha\ast|u|^p)|u|^pdx.
\end{split}
\end{equation}

We use  the Hardy-Littlewood-Sobolev inequality, the interpolation
inequality, the conservation laws, the Sobolev imbedding theorem and
Lemma \ref{lem Radial Sob} to estimate
\begin{equation}\label{e1.13}
\begin{split}
&\int_{|x|>R}(I_\alpha\ast|u|^p)|u|^pdx\\
&=\mathcal{A}\int_{|x|>R}\int_{|y|>\frac{R}{2}}\frac{|u(x)|^p|u(y)|^p}{|x-y|^{N-\alpha}}dydx
+\mathcal{A}\int_{|x|>R}\int_{|y|<\frac{R}{2}}\frac{|u(x)|^p|u(y)|^p}{|x-y|^{N-\alpha}}dydx\\
&:=I+II,
\end{split}
\end{equation}
where
\begin{equation}\label{e8.11}
\begin{split}
I&\lesssim
\||u|^p\chi_{B_{{R}/{2}}^c}\|_{L^{\frac{2N}{N+\alpha}}}\||u|^p\chi_{B_{R}^c}\|_{L^{\frac{2N}{N+\alpha}}}\\
&\lesssim
\|u\|_{L^2}^{\frac{N+\alpha}{N}}\|u\|_{L^{\infty}(|x|>\frac{R}{2})}^{p-\frac{N+\alpha}{N}}
\|u\|_{L^2}^{\frac{N+\alpha}{N}}\|u\|_{L^{\infty}(|x|>R)}^{p-\frac{N+\alpha}{N}}\\
&\lesssim
R^{-(N-1)(p-\frac{N+\alpha}{N})}\|u\|_{\dot{H}_b^1}^{p-\frac{N+\alpha}{N}}
\end{split}
\end{equation}
and
\begin{equation}\label{e8.10}
\begin{split}
II&\lesssim R^{-(N-\alpha)}\int_{|y|<\frac{R}{2}}|u(y)|^pdy
\int_{|x|>R}|u(x)|^pdx\\
&\lesssim
R^{-(N-\alpha)}\|u\|_{L^2}^{(1-\theta)p}\|u\|_{L^{2^*}}^{p\theta}\|u\|_{L^2}^2\|u\|_{L^{\infty}(|x|>R)}^{p-2}\\
&\lesssim
R^{-(N-\alpha)}\|u\|_{\dot{H}_b^1}^{p\theta}R^{-\frac{N-2s}{2}(p-2)}\left(\sup_{|x|>R}|x|^{\frac{N-2s}{2}}|u(x)|\right)^{p-2}\\
&\lesssim
R^{-(N-\alpha)-\frac{N-2s}{2}(p-2)}\|u\|_{\dot{H}_b^1}^{p\theta+(p-2)s}\\
&\lesssim
R^{-(N-\alpha+\frac{N-1}{2}(p-2))}\|u\|_{\dot{H}_b^1}^{\frac{N+1}{2}(p-2)}
\end{split}
\end{equation}
with $s=\frac{1}{2}$ and
$\frac{1}{p}=\frac{1-\theta}{2}+\frac{\theta}{2^*}$. Inserting
(\ref{e5.344})-(\ref{e8.10}) into (\ref{e5.341}),  we complete the
proof.
\end{proof}

\textbf{5.2.  Blowup}.

\textbf{Proof of the blowup part (ii) in Theorem \ref{thm1.3}}. By
using the standard virial identity (Lemma \ref{lem6.2}) and the
conservation laws, we have
\begin{equation*}
\frac{d^2}{dt^2}\|xu(t)\|_{L^2}^2=16E_b(u(t))=16E_b(u_0)<0.
\end{equation*}
By the standard convexity arguments (see \cite{Glassey 1977}), we
know that $u$ blows up in finite time in both time directions.

\textbf{Proof of the blowup part (ii) in Theorem \ref{thm1.4}}. We
proceed as in the proof of Theorem \ref{thm1.4} (i). It follows from
the assumption
\begin{equation}\label{e7.1}
E_b(u_0)\|u_0\|_{L^2}^{2\sigma}<H(b)
\end{equation}
that there exists $\delta_1>0$ small enough such that
\begin{equation}\label{e7.6}
E_b(u_0)\|u_0\|_{L^2}^{2\sigma}<(1-\delta_1)H(b),
\end{equation}
which combines with (\ref{e6.1}) and the conservation laws implies
that
\begin{equation}\label{e7.5}
f(\|u(t)\|_{\dot{H}_b^1}\|u(t)\|_{L^2}^{\sigma})\leq
E_b(u(t))\|u(t)\|_{L^2}^{2\sigma}< (1-\delta_1)H(b)\ \mathrm{for\
any\ }t\in I.
\end{equation}
Since $f(K(b))=H(b)$,
$\|u_0\|_{\dot{H}_b^1}\|u_0\|_{L^2}^{\sigma}>K(b)$, in view of
(\ref{e7.5}), and the continuity argument, there exists $\delta_2>0$
depending on $\delta_1$ such that
\begin{equation}\label{e7.7}
\|u(t)\|_{\dot{H}_b^1}\|u(t)\|_{L^2}^{\sigma}>(1+\delta_2)K(b)\
\mathrm{for\ any\ }t\in I.
\end{equation}

Next we claim that for $\epsilon>0$ small enough, there exists $c>0$
such that
\begin{equation}\label{e7.8}
(8+\epsilon)\|u(t)\|_{\dot{H}_b^1}^2+\frac{4\alpha+4N-4Np}{p}\int_{\mathbb{R}^N}(I_\alpha\ast|u|^p)|u|^pdx
\leq -c
\end{equation}
for any $t\in I$. Indeed, multiplying the left side of (\ref{e7.8})
by $\|u(t)\|_{L^2}^{2\sigma}$ and using (\ref{e7.5}), (\ref{e7.7}),
the conservation laws and
$H(b)=\frac{Np-N-\alpha-2}{2(Np-N-\alpha)}K^2(b)$, we get that
\begin{equation*}
\begin{split}
&\mathrm{LHS}(\ref{e7.8})\times \|u(t)\|_{L^2}^{2\sigma}\\
=&
8(Np-N-\alpha)E_b(u)\|u(t)\|_{L^2}^{2\sigma}+(8+4\alpha+4N-4Np+\epsilon)\|u(t)\|_{\dot{H}_b^1}^2\|u(t)\|_{L^2}^{2\sigma}\\
\leq &
8(Np-N-\alpha)(1-\delta_1)H(b)+(8+4\alpha+4N-4Np+\epsilon)(1+\delta_2)^2K(b)^2\\
=&\left[(4Np-4N-4\alpha-8)(1-\delta_1-(1+\delta_2)^2)+\epsilon(1+\delta_2)^2\right]K(b)^2\\
<&-c
\end{split}
\end{equation*}
by choosing $\epsilon>0$ small enough. Hence, the claim holds.

Case 1 ($xu_0\in L^2(\mathbb{R}^N)$). By using the standard virial
identity (Lemma \ref{lem6.2}) and (\ref{e7.8}), we have
\begin{equation*}
\frac{d^2}{dt^2}\|xu(t)\|_{L^2}^2=8\|u(t)\|_{\dot{H}_b^1}^2
+\frac{4\alpha+4N-4Np}{p}\int_{\mathbb{R}^N}(I_\alpha\ast|u|^p)|u|^pdx\leq
-c,
\end{equation*}
which implies that $u$ blows up in finite time.

Case 2 ($u_0\in H_r^1(\mathbb{R}^N)$). Since
$p<\min\{p^b,\frac{2N+6}{N+1}\}$, we have $p-\frac{N+\alpha}{N}<2$
and $\frac{N+1}{2}(p-2)<2$. By using the local virial identity
(Lemma \ref{lem6.3}), the Young inequality (\ref{e2.3}) and
(\ref{e7.8}), for any $\epsilon>0$,
 we have
\begin{equation*}
\begin{split}
\frac{d^2}{dt^2}\|\psi_R(x)u\|_{L^2}^2&\leq
8\|u\|_{\dot{H}_b^1}^2+\frac{4\alpha+4N-4Np}{p}\int_{\mathbb{R}^N}(I_\alpha\ast|u|^p)|u|^pdx\\
&\quad +
O(R^{-2})+O(R^{-(N-1)(p-\frac{N+\alpha}{N})}\|u\|_{\dot{H}_b^1}^{p-\frac{N+\alpha}{N}})\\
&\quad+
O(R^{-(N-\alpha+\frac{N-1}{2}(p-2))}\|u\|_{\dot{H}_b^1}^{\frac{N+1}{2}(p-2)})\\
&\leq
(8+\epsilon)\|u(t)\|_{\dot{H}_b^1}^2
+\frac{4\alpha+4N-4Np}{p}\int_{\mathbb{R}^N}(I_\alpha\ast|u|^p)|u|^pdx\\
&\quad+O(R^{-2})+O(\epsilon^{-\frac{p-\frac{N+\alpha}{N}}{2-(p-\frac{N+\alpha}{N})}}R^{-(N-1)(p-\frac{N+\alpha}{N})\frac{2}{2-(p-\frac{N+\alpha}{N})}})\\
&\quad+ O(\epsilon^{-\frac{\frac{N+1}{2}(p-2)}{2-\frac{N+1}{2}(p-2)}}R^{-(N-\alpha+\frac{N-1}{2}(p-2))\frac{2}{2-\frac{N+1}{2}(p-2)}})\\
&\leq -c/2
\end{split}
\end{equation*}
by choosing $\epsilon>0$ small enough and by choosing $R>1$  large
enough depending on $\epsilon$. Hence, the solution $u$ blows up in
finite time.

\bigskip

\textbf{Proof of (1) in  Remark \ref{rem1.1}}. Let $E_b>0$, we find
initial value $u_0\in H^1$ with $E_b(u_0)=E_b$ such that the
corresponding solution $u$ blows up in finite time. We follow the
standard argument (see Remark 6.5.8 in \cite{Cazenave 2003}). Using
the virial identity with $p=p_b$, we have
\begin{equation*}
\frac{d^2}{dt^2}\|xu(t)\|_{L^2}^2=16 E_b(u_0).
\end{equation*}
Hence,
\begin{equation*}
\|xu(t)\|_{L^2}^2=\|xu_0\|_{L^2}^2+4t\mathrm{Im}\int_{\mathbb{R}^N}\overline{u_0}x\cdot\nabla
\overline{u_0}dx+8t^2 E_b(u_0):=f(t).
\end{equation*}
Note that if $f$ takes negative values, then the solution $u$ must
blow up in finite time. In order to make $f$ takes negative values,
we need
\begin{equation}\label{e8.2}
\left(\mathrm{Im}\int_{\mathbb{R}^N}\overline{u_0}x\cdot\nabla
\overline{u_0}dx\right)^2>2E_b(u_0)\|xu_0\|_{L^2}^2.
\end{equation}
Now fix $\theta\in C_0^{\infty}(\mathbb{R}^N)$ a real-valued
function and set $\psi(x)=e^{-i|x|^2}\theta(x)$. We see that
$\psi\in C_0^{\infty}(\mathbb{R}^N)$ and
\begin{equation*}
\mathrm{Im}\int_{\mathbb{R}^N}\bar{\psi}x\cdot\nabla
\bar{\psi}dx=-2\int_{\mathbb{R}^N}|x|^2|\theta|^2dx<0.
\end{equation*}

Denote
\begin{equation*}
\begin{split}
&A=\frac{1}{2}\|\psi\|_{\dot{H}_b^1}^2,\
B=\frac{1}{2p}\int_{\mathbb{R}^N}(I_\alpha\ast|\psi|^p)|\psi|^pdx,\\
&C=\|x\psi\|_{L^2}^2,\
D=-\mathrm{Im}\int_{\mathbb{R}^N}\bar{\psi}x\cdot\nabla
\bar{\psi}dx.
\end{split}
\end{equation*}
Then $A,\ B,\ C,\ D>0$. For $\lambda, \mu>0$, set
$u_0(x)=\lambda\psi(\mu x)$.  By direct calculation, we have
\begin{equation*}
E_b(u_0)=\lambda^2\mu^{2-N}A-\lambda^{2p}\mu^{-N-\alpha}B
=\lambda^2\mu^{2-N}\left(A-\frac{\lambda^{2p-2}}{\mu^{2+\alpha}}B\right),
\end{equation*}
\begin{equation*}
\|xu_0\|_{L^2}^2=\lambda^{2}\mu^{-2-N}C,
\end{equation*}
\begin{equation*}
\mathrm{Im}\int_{\mathbb{R}^N}\overline{u_0}x\cdot\nabla
\overline{u_0}dx=-\lambda^2\mu^{-N}D.
\end{equation*}
Next,  we  choose $\lambda$ and $\mu$ such that $E_b(u_0)=E_b$ and
(\ref{e8.2}) holds. Hence,
\begin{equation}\label{e8.3}
\lambda^2\mu^{2-N}\left(A-\frac{\lambda^{2p-2}}{\mu^{2+\alpha}}B\right)=E_b
\end{equation}
and
\begin{equation}\label{e8.4}
\frac{D^2}{C}>2\left(A-\frac{\lambda^{2p-2}}{\mu^{2+\alpha}}B\right).
\end{equation}
Fix $0<\epsilon<\min\{A,\frac{D^2}{2C}\}$ and choose
\begin{equation}\label{e8.5}
\frac{\lambda^{2p-2}}{\mu^{2+\alpha}}B=A-\epsilon.
\end{equation}
It is obvious that (\ref{e8.4}) is satisfied. Condition (\ref{e8.3})
and (\ref{e8.5}) imply
\begin{equation}\label{e8.6}
\lambda^2\mu^{2-N}\epsilon=E_b.
\end{equation}
So we can solve $\mu$ and $\lambda$ from (\ref{e8.5}) and
(\ref{e8.6}). The proof is complete.

\bigskip

\textbf{Acknowledgements.} This work was supported by the research
project of Tianjin education commission with the Grant no. 2017KJ173
``Qualitative studies of solutions for two kinds of nonlocal
elliptic equations".



\end{document}